\documentclass[a4paper]{amsart}
\usepackage{graphicx,amsmath,amsfonts,latexsym,amssymb,amsthm,mathrsfs, color}
\usepackage[latin1]{inputenc}
\usepackage[T1]{fontenc}
\DeclareMathOperator{\op}{op}
\newcommand{\B}{{\mathbb B}}
\newcommand{\C}{{\mathbb C}}

\newcommand{\N}{{\mathbb N}}

\newcommand{\R}{{\mathbb R}}

\newcommand{\cC}{{\mathcal C}}
\newcommand{\cD}{{\mathcal D}}
\newcommand{\cE}{{\mathcal E}}
\newcommand{\cF}{{\mathcal F}}

\newcommand{\cH}{{\mathcal H}}
\newcommand{\cK}{{\mathcal K}}
\newcommand{\cL}{{\mathcal L}}

\newcommand{\cO}{{\mathcal O}}
\newcommand{\cP}{{\mathcal P}}

\newcommand{\cS}{{\mathcal S}}

\newcommand{\ga}{\alpha}
\newcommand{\gb}{\beta}
\renewcommand{\gg}{\gamma}
\newcommand{\gG}{\Gamma}
\newcommand{\gd}{\delta}
\newcommand{\gD}{\Delta}

\newcommand{\gve}{\varepsilon}

\newcommand{\gl}{\lambda}
\newcommand{\gL}{\Lambda}
\newcommand{\go}{\omega}
\newcommand{\gO}{\Omega}

\newcommand{\gt}{\theta}

\newcommand{\gs}{\sigma}

\newcommand{\Proof}[1]{{\em Proof}. #1~\hfill$\Box$\medskip}

\renewcommand{\Re}{\mathop{\rm Re}}
\renewcommand{\Im}{\mathop{\rm Im}}

\newtheorem{thm}{Theorem}[section]
\newtheorem{proposition}[thm]{Proposition}
\newtheorem{corollary}[thm]{Corollary}
\newtheorem{definition}[thm]{Definition}
\newtheorem{theorem}[thm]{Theorem}
\newtheorem{lemma}[thm]{Lemma}
\newtheorem{remark}[thm]{Remark}

\begin{document}
\title[Bounded Imaginary Powers of Cone Differential Operators]%
{Bounded Imaginary Powers of Cone Differential Operators\\ on Higher Order 
Mellin-Sobolev Spaces \\
and Applications to the Cahn-Hilliard Equation}
\author{Nikolaos Roidos}
\author{Elmar Schrohe}
\address{Institut für Analysis, Leibniz Universität Hannover, Welfengarten 1, 
30167 Hannover, Germany}
\email{roidos@math.uni-hannover.de, schrohe@math.uni-hannover.de}

\begin{abstract}
Extending earlier results on the existence of bounded imaginary powers for cone differential operators on weighted $L^p$-spaces $\cH^{0,\gg}_p(\B)$ over a manifold with conical 
singularities, we show how the same assumptions also yield the existence of bounded imaginary powers on higher order Mellin-Sobolev spaces $\cH^{s,\gg}_p(\B)$, $s\ge0$. 

As an application we consider the Cahn-Hilliard equation on a manifold with (possibly warped) conical singularities. 
Relying on our work for the case of straight cones, we first establish $R$-sectoriality (and thus maximal regularity) for the linearized equation and then deduce the existence of a short time solution with the help of a theorem by Clément and Li. We also obtain the short time asymptotics of the solution near the conical point.
\end{abstract}
\date{\today}

\maketitle

\section{Introduction}
In this article we show the existence of bounded imaginary powers for a class of elliptic differential operators on higher order Mellin-Sobolev spaces over manifolds with conical singularities and use this to establish the existence of
maximally regular solutions to the Cahn-Hilliard equation.

We model the underlying manifold with conical singularities by a smooth manifold $\B$ with boundary, of dimension $n+1$, $n\ge1$, endowed with a conically degenerate Riemannian metric. 
In a collar neighborhood of the boundary, we choose coordinates $(x,y)$ with $0\le x<1$ and
$y\in \partial \B$. The metric then is assumed to be of the form 
\begin{eqnarray}\label{e2}
g = dx^2+ x^2 h(x),
\end{eqnarray}
where $x\mapsto h(x)$ is a smooth family of Riemannian metrics on the cross-section,
non-degenerate up to $x=0$. 
The cone is straight, if this family is constant near $x=0$; otherwise it is warped. 

We call a differential operator $A$ of order $\mu$ on the interior $\B^\circ$ of $\B$ a
cone differential operator or conically degenerate, if, in local coordinates near the boundary, it can be written in the form 
\begin{eqnarray}\label{conediffop}
A=x^{-\mu}\sum_{j=0}^\mu a_j(x,y,D_y) (-x\partial_x)^j
\end{eqnarray} 
with $x\mapsto a_j(x,y,D_y)$ smooth families, up to $x=0$, of differential operators of order 
$\mu-j$ on the cross-section $\partial \B$. 
It is degenerate elliptic or shortly $\B$-elliptic, if the principal pseudodifferential symbol $\gs^\mu_\psi(A)$ is invertible on $T^*\B^\circ$ and, moreover, in local coordinates 
$(x,y)$ near the boundary and corresponding covariables 
$(\xi,\eta)$, the rescaled principal symbol  
\begin{eqnarray}\label{rps}
x^{\mu}\gs^\mu_\psi(A)(x,y,\xi/x,\eta) 
= \sum_{j=0}^\mu\gs^{\mu-j}_\psi a_j(x,y,\eta)(-i\xi)^j
\end{eqnarray}
is invertible up to $x=0$.

The Laplace-Beltrami operator with respect to the metric $g$ above is an elliptic cone differential operator. 
In fact, a short computation shows that, in local coordinates near the boundary, it is of the form 
\begin{eqnarray}\label{Delta}
\Delta = 
x^{-2}\left((x\partial_x)^2  
- (n-1+H(x))(-x\partial_x )
+ \Delta_{h(x)}
 \right),
\end{eqnarray}
where $\Delta_{h(x)}$ is the Laplace-Beltrami operator on the cross-section with respect to the metric $h(x)$ and $H(x) = x\partial_x (\det h(x))/(2\det h(x))$.

Conically degenerate operators act in a natural way on scales of weighted Mellin-Sobolev spaces 
$\cH^{s,\gg}_p(\B)$, $s,\gg\in\R$, $1<p<\infty$, cf.\ Section \ref{MSS}. 
For $s\in \N_0$, $\cH^{s,\gg}_p(\B)$ is the space of all functions $u$ in $H^s_{loc}(\B^\circ)$
such that, near the boundary, 
\begin{eqnarray}\label{measure}
x^{\frac{n+1}2-\gamma}(x\partial_x)^j\partial_y^{\alpha}u(x,y)
\in L^p\Big([0,1]\times \partial \B, \sqrt{\det h(x)}\,\frac{dx}xdy\Big),\quad j+|\ga|\le s.
\end{eqnarray}
In particular, $\cH^{0,\gg}_p(\B)$ simply is a weighted $L^p$-space. 
The factor $\sqrt{\det h(x)}$ could be omitted in the definition of the spaces; 
it will be important only for questions of symmetry of the (warped) Laplacian. 

Sometimes it is necessary to consider a cone differential operator $A$ as an unbounded operator in a fixed space $\cH^{s,\gg}_p(\B)$. 
If $A$ is $\B$-elliptic, it turns out that, under a mild additional assumption, the domain of the minimal extension, i.e. the closure of $A$ considered as an operator on $C^\infty_c(\B^\circ)$, is 
$$\cD(A_{\min,s})= \cH^{s+\mu,\gg+\mu}_p(\B),$$ 
while the domain of the maximal extension is
$$\cD(A_{\max,s}) = \{u\in \cH^{s,\gg}_p(\B):  Au\in \cH^{s,\gg}_p(\B)\} = 
\cD(A_{\min,s}) \oplus \cE,$$
where $\cE$ is a finite-dimensional space consisting of linear combinations of functions of the form $x^{-\rho}\log^k x \, c(y)$ with $\rho\in \C$, $k\in \N_0$ and a smooth function $c$ on the cross-section. It can be chosen independent of $s$. This result has a long history, see e.g. \cite{BrueningSeeley88}, \cite{Le},  \cite{Sh},  \cite{se}; the present version is due to Gil, Krainer and Mendoza \cite{GKM}.
 
A densely defined unbounded operator $A$ on a Banach space $X_0$ is said to have bounded imaginary powers with angle $\phi\ge0$, 
provided that 
(i) its resolvent exists in a closed sector $\gL_\theta$ of angle $\theta$ around the negative real axis and decays there like $\gl^{-1}$ as $\gl\to\infty$, and (ii) the purely imaginary powers $A^{it}$, $t\in \R$, satisfy the estimate
$$\|A^{it}\|_{\cL(X_0)} \le M e^{\phi |t|}$$
for a suitable constant  $M$. 

It was shown in \cite{CSS2}, that $\B$-elliptic cone differential operators have bounded imaginary powers on $\cH^{0,\gg}_p(\B)$, provided their resolvent exists in a sector $\gL_\theta$ as above and has a certain structure. For details see Theorem 2 in connection with Remark 6 in the cited article.  

In the present paper, we show that the same assumptions 
on the structure of the resolvent also guarantee the existence of bounded imaginary powers on higher order Mellin-Sobolev spaces $\cH^{s,\gg}_p(\B)$, $s\ge0$.

This comes as a small surprise. Of course, the boundedness of imaginary powers of 
elliptic pseudodifferential operators on $L^p(M)$, where $M$ is a closed manifold, carries over to all Sobolev spaces $H_p^s(M)$, $s\in \R$. 
But already boundary value problems exhibit a
more complicated behavior. Seeley \cite{See} proved the existence of bounded
imaginary powers for certain elliptic boundary value problems in $L^p$.  
A direct computation, cf.\ Nesensohn  \cite{Nes}, however, shows that
the resolvent to the Dirichlet-Laplacian $\gD_{\text{Dir}}$, considered as an unbounded operator in the Sobolev space $H^1_p(\gO)$ over a bounded domain $\gO$ with domain 
$\cD(\gD_{\text{Dir}})=\{u\in H^3_p(\gO): u|_{\partial \gO}=0\}$, decays only as 
$|\gl|^{-1/2-1/(2p)}$ as $|\gl|\to\infty$, so that already the most basic condition for 
the boundedness of imaginary powers is violated. 
See Denk and Dreher \cite{DD} for a careful analysis of the 
resolvent decay for boundary value problems.
 
For cone differential operators, Gil, Krainer and Mendoza studied the resolvent decay on higher order Mellin-Sobolev spaces. 
In  \cite[Theorem 6.36]{GKM}, they showed that,
under certain conditions implying existence and $O(|\gl|^{-1})$-decay of the resolvent on 
$\cH^{0,\gg}_2(\B)$ for $\gl$ in a sector of the complex plane,  the resolvent also exists on $\cH^{s,\gg}_2(\B)$. 
Moreover, it is $O(|\gl|^{M(s)})$ as $|\gl|\to\infty$ in the sector, with a function $M(s)$ of at most polynomial growth. 

It was therefore not to be expected that cone differential operators satisfying the assumptions in \cite{CSS2} would have bounded imaginary powers on $\cH^{s,\gg}_p(\B)$, $s\ge0$.  
In fact, our proof reduces the higher order case to the case $s=0$ with the help of 
commutator techniques. %ES

We use the result on bounded imaginary powers in order to improve earlier work on 
the Cahn-Hilliard equation on manifolds with conical singularities.

The Cahn-Hilliard equation 
\begin{eqnarray}\label{CH1}
\partial_t u(t)+\Delta^{2}u(t)+\Delta\big(u(t)-u^{3}(t)\big)&=&0, \quad  t\in(0,T);
\\ 
u(0)&=&u_{0},\label{CH2}
\end{eqnarray}
is a phase-field or diffuse interface equation which models phase separation of a binary mixture; $u$ denotes the concentration difference of the components. 
The sets where $u=\pm1$ correspond to domains of pure phases.
Global solvability had been established before, cf.\  Elliott and Zheng Songmu \cite{ElliottZhengSongmu86} or Caffarelli and Muler \cite{CaffarelliMuler95}.
We are interested in the properties of the solutions caused by the conical singularities. %ES 

In \cite{RS} we studied the case of a {\em straight} conical singularity.
Our analysis was based on understanding the Laplacian $\gD$ associated with the cone metric. 
We fixed a suitable extension $\underline \gD$ of $\gD$ in $\cH^{0,\gg}_p(\B)$ and chose the domain of the bilaplacian correspondingly as 
$$\cD(\underline \gD^2)=\{u\in \cD(\underline \Delta): \Delta u \in \cD(\underline \gD)\}.$$
We could then show the existence of a maximally regular short time solution to the Cahn-Hilliard equation. An important ingredient was the existence of bounded imaginary powers for $\underline \gD$ established in \cite{CSS2}.

We  extend  this  in two directions: 

\begin{itemize}
\item  For the case of straight conical singularities we consider extensions of $\gD$ in the higher order Mellin-Sobolev spaces $\cH^{s,\gg}_p(\B)$, $s\ge0$, %ES
and establish higher regularity of the solutions using the above mentioned result on bounded imaginary powers.

\item We treat the case, where the manifold has warped conical singularities 
with the help of $R$-sectoriality and suitable perturbation results. 
\end{itemize}
Our strategy is to prove first maximal regularity for the linearized equation. 
For manifolds with straight conical singularities, we obtain it from the existence of bounded imaginary powers for the Laplacian by a combination of the above results with \cite[Theorem 5.7]{Sh}, cf.\ Theorem \ref{hibip}, below. 

For manifolds with warped cones, the existence of bounded imaginary powers for the Laplacian is not clear. 
We instead infer maximal regularity from $R$-sectoriality and perturbation results by 
Kunstmann and Weis \cite{KL}. 
Maximal regularity together with a theorem by Clément and Li then implies %ES
the existence of a short time solution to the non-linear equation. We obtain the following result: \medskip

{\bf Theorem.} {\it
Let $\underline \gD$ be the closed extension of the cone Laplacian in $\cH^{s,\gg}_p(\B)$ with domain $\cD(\underline \gD) = \cH^{s+2,\gg+2}_p(\B)\oplus \C$. 
Then, for small $T$, the Cahn-Hilliard equation with initial value $u_0$ has a unique solution $u$ in the space}
\begin{eqnarray*}
 W^{1}_q(0,T; \cH^{s,\gg}_p(\B))\cap L^q(0,T; \cD(\underline\gD^2)). 
\end{eqnarray*}
Here we choose $p\ge\dim \B$, $q>2$ and $\gg$ slightly larger than $\dim\B/2-2$; the possible choices are limited by spectral data for the boundary Laplacian $\gD_{h(0)}$. 
%In the case of straight conical singularities we can take $q>2$; for warped cones,  $q\in (2,\infty)$ has to be taken sufficiently large. 
The initial value $u_0$ necessarily is an element of the real interpolation space 
$X_q=(\cD(\underline\gD^2),\cH^{s,\gg}_p(\B))_{1/q,q}$, and the solution belongs
to $C([0,T]; X_q)$. 

As the domain of $\gD^2$ (under a mild additional assumption) equals $\cH^{s+4,\gg+4}_p(\B)\oplus \cF$, where  $\cF$ is a finite-dimensional space of singular functions of the form $x^q\log^k x\, c(y)$,  the theorem gives us information on the asymptotics of the solution close to the conic point. 
The asymptotics can be determined rather explicitly. For straight cones, they only depend on spectral data of $\gD_{h(0)}$; see Section \ref{domd2} for details. 
On warped cones; they additionally depend on the metric $h$ and its derivatives in $x=0$. 
 
In principle, maximal regularity and the theorem of Clément and Li imply solvability 
even for quasilinear equations, while the Cahn-Hilliard equation is only semilinear. 
Still, this equation already exhibits very clearly the difficulties that arise as a consequence of the combination of singular analysis with nonlinear theory while many computations still can be performed explicitly, so that it seems a good example. 

Our results are complemented by a recent article by Vertman \cite{V}. 
He studied the Cahn-Hilliard  (and more general) equations on manifolds with edges, showed the existence of short time solutions and obtained results on their short time asymptotics. 
He works in the $L^2$-setting with the analysis based on the Friedrichs extension of the Laplacian and a microlocal heat kernel construction. 
In \cite[Definition 2.2(i)]{V} a spectral condition on $\gD_{h(0)}$ is imposed which, in the language of this article, corresponds to the assumption that $\overline \gve >1$ in \eqref{epsilonbar}. 
Moreover, warping of the cone is admitted, but condition (ii) of the mentioned definition requires the difference to the straight metric to be $O(x^2)$. 

This article is structured as follows: In Section 2, we recall basic notions such as
bounded imaginary powers, $R$-sectoriality and maximal regularity.
We prove the existence of bounded imaginary powers in Section 3.  
Section 4 contains the analysis of the Cahn-Hilliard equation in higher order Mellin-Sobolev spaces for the case of a straight cone metric. The case of warped cones is treated in Section 5. 
 
%%%%%%%%%%%%%%%%%%%%%%%%%%%%%%%%%%%%%%%%%%
%%%%%%%%%%%%%%%%%%%%%%%%%%%%%%%%%%%%%%%%%%

\section{Preliminary Results on Parabolic Problems and Mellin Sobolev Spaces}
\setcounter{equation}{0}
\subsection{Maximal $L^p$-regularity and parabolic problems}
We start with the notion of sectoriality which guarantees the existence of solutions for the linearized problem. For the rest of the section let $X_0$  be a Banach space.

\begin{definition}\label{sectorial}
For $\theta\in[0,\pi[$ denote by  $\mathcal{P}(\theta)$  the class of all closed densely defined linear operators $A$ in $X_0$ such that 
\[
S_{\theta}=\{z\in\mathbb{C}\,|\, |\arg z|\leq\theta\}\cup\{0\}\subset\rho{(-A)} \quad \mbox{and} \quad (1+|z|)\|(A+z)^{-1}\|\leq K_{\theta}, \quad z\in S_{\theta},
\]
for some $K_{\theta}\geq1$. The elements in $\mathcal{P}(\theta)$ are called {\em sectorial operators of angle $\theta$}.
\end{definition}

The Dunford integral allows us to define the complex powers of a sectorial operator for negative real part. The definition then extends to give arbitrary complex powers, cf. Amann \cite[III.4.6.5]{Am}.

\begin{definition}
Let $A\in\mathcal{P}(\theta)$, $\theta\in[0,\pi[$. We say that $A$  has {\em bounded imaginary powers} if there exists some $\varepsilon>0$ and $K\geq1$ such that 
\[
A^{it}\in\mathcal{L}(X_0) \,\,\,\,\,\, \mbox{and} \,\,\,\,\,\, \|A^{it}\|\leq K \,\,\,\,\,\, \mbox{for all} \,\,\,\,\,\, t\in[-\varepsilon,\varepsilon].
\]
In this case, there exists a $\phi\geq0$ such that  $\|A^{it}\|\leq M e^{\phi|t|}$ for all $t\in\mathbb{R}$ with some $M\geq1$, and  we write $A\in\mathcal{BIP}(\phi)$.
\end{definition}

We continue with the notion of the $R$-sectoriality, 
that is slightly stronger than the standard sectoriality and will guarantee maximal regularity for the linearized problem. 

\begin{definition}
Let $\theta\in[0,\pi[$. An operator $A\in\mathcal{P}(\theta)$ is called {\em $R$-sectorial of angle $\theta$} if for any choice of $\lambda_{1},...,\lambda_{n}\in S_{\theta}$,  $x_{1},...,x_{n}\in X_0$, and $n\in\mathbb{N}$, we have 
\[
\big\|\sum_{k=1}^{n}\epsilon_{k}\lambda_{k}(A+\lambda_{k})^{-1}x_{k}\big\|_{L^{2}(0,1;X_0)} \leq C \big\|\sum_{k=1}^{n}\epsilon_{k}x_{k}\big\|_{L^{2}(0,1;X_0)},
\]
for some constant $C\geq 1$, called the $R$-bound, and the sequence  $\{\epsilon_{k}\}_{k=1}^{\infty}$  of the Rademacher functions. 
\end{definition} 

Let $A$ be a closed densely defined linear operator $A:\mathcal{D}(A)=X_{1}\rightarrow X_{0}$. 
It is well known that $-A$ generates a bounded analytic semigroup if and only if $c+A\in\mathcal{P}(\theta)$ for some $c\in \C$ and some $\theta>\pi/2$. For such $A$, consider the  Cauchy problem
\begin{gather}\label{AP}
\Big\{\begin{array}{lclc} u'(t)+Au(t)=g(t), & t\in(0,T)\\
u(0)=u_{0}&
\end{array} 
\end{gather}
in the $X_{0}$-valued $L^{q}$-space $L^{q}(0,T;X_{0})$, where $1<q<\infty$,  $T>0$.
We say that $A$ has {\em maximal $L^{q}$-regularity}, if for some $q$ (and hence by a result of Dore \cite{Do} for all) we have that, given any data $g\in L^{q}(0,T;X_0)$ and $u_{0}$ in the real interpolation space $X_{q}=(X_{1},X_{0})_{\frac{1}{q},q}$, the unique solution of \eqref{AP} belongs to $L^{q}(0,T;X_{1})\cap W^{1,q}(0,T;X_{0})\cap C([0,T];X_{q})$ and depends continuously on $g$ and $u_{0}$. 
If the space $X_{0}$ is UMD (unconditionality of martingale differences property) then the following result holds.

\begin{theorem}{\rm (Weis, \cite[Theorem 4.2]{W})}
In a UMD Banach space any $R$-sectorial operator of angle greater than $\frac{\pi}{2}$ has maximal $L^{q}$-regularity. 
\end{theorem}

\begin{remark}\rm 
In a UMD space, an operator $A\in \mathcal{BIP}(\phi)$ with $\phi<\pi/2$ is
$R$-sectorial  with angle greater than $\pi/2$ by \cite[Theorem 4]{CP} and hence has maximal $L^q$-regularity. This also is a classical result by Dore and Venni \cite{DV}.  
\end{remark}

Next, we consider quasilinear problems of the form
\begin{gather}\label{QL}
\Big\{\begin{array}{lclc} \partial_{t}u(t)+A(u(t))u(t)=f(t,u(t))+g(t), & t\in(0,T_{0});\\
u(0)=u_{0}&	
\end{array} 
\end{gather}
in $L^{q}(0,T_{0};X_{0})$, such that $\mathcal{D}(A(u(t)))=X_{1}$, $1<q<\infty$ and $T_{0}$ is finite. The main tool we use for proving the existence of solutions of the above problems and regularity results is the following theorem that is based on a Banach fixed point argument.  

\begin{theorem}\label{CL} {\rm (Clément and Li, \cite[Theorem 2.1]{CL}) }
Assume that  there exists an open neighborhood $U$ of $u_0$ in $X_q$ such that $A(u_0): X_{1}\rightarrow X_{0}$ has maximal $L^{q}$-regularity and that  
\begin{itemize}
\item[(H1)] $A\in C^{1-}(U, \cL(X_1,X_0))$,    
\item[(H2)] $f\in C^{1-,1-}([0,T_0]\times U, X_0)$,
\item[(H3)] $g\in L^q(0,T_0; X_0)$.
\end{itemize}
Then there exists a $T>0$ and a unique $u\in L^q(0,T;X_1)\cap W^1_q(0,T;X_0)\cap C([0,T];X_q)$ solving the equation \eqref{QL} on $(0,T)$.
\end{theorem}

\subsection{Mellin Sobolev Spaces}\label{MSS}

By a cut-off function (near $\partial \B$) we mean a smooth non-negative
function $\go$ on $\B$ with $\go\equiv1$ near $\partial \B$  and $\go\equiv0$ 
outside the collar neighborhood  of the boundary. 

There are various ways of extending the definition of the Mellin-Sobolev spaces 
$\cH^{s,\gg}_p(\B)$ given in the introduction for $s\in \N$ to arbitrary $s\in\R$.
A simple way, cf.~\cite{CSS1}, is given via the map 
$$\cS_\gamma: C^\infty_c(\R^{n+1}_+) \to   C^\infty_c(\R^{n+1}),\qquad 
v(t,y) \mapsto e^{(\gg-\frac{n+1}2)t} v(e^{-t},y).$$
Let $\kappa_{j}:U_{j}\subseteq\partial \B\to\R^n$, $j=1,\ldots,N,$
be a covering of $\partial \B$ by coordinate charts and $\{\varphi_{j}\}$ 
a subordinate partition of unity. 

\begin{definition}
$\cH^{s,\gg}_p(\B)$, $s,\gg\in\R$, $1<p<\infty$, is the space of all distributions on $\B^\circ$ such that 
    \begin{equation}\label{norm}
	\begin{array}{lcl}
	    \|u\|_{\cH^{s,\gg}_p(\B)} \! & \! = \! & \!
	    \displaystyle
	     \sum_{j=1}^{N}\|\cS_{\gamma}(1\otimes\kappa_{j})_{*}
	              (\omega\varphi_{j}u)\|_{H^{s}_{p}(\R^{1+n})}+
	             \| (1-\omega)u)\|_{H^{s}_{p}(\B)}
        \end{array}    
    \end{equation}
is defined and finite. Here, $\omega$ is a (fixed) cut-off 
function and $*$ refers to the push-forward of distributions. 
Up to equivalence of norms, this construction is independent of the choice of 
$\go$ and the $\kappa_j$. 
Clearly, all the spaces $\cH^{s,\gg}_p(\B)$ are UMD spaces. 
\end{definition}

\begin{corollary}\label{c0}Let $s>(n+1)/p$. 
Then $\|uv\|_{\cH^{s,\gg}_p(\B)}\le \|u\|_{\cH^{s,\gg}_p(\B)}\|v\|_{\cH^{s,(n+1)/2}_p(\B)}$.  
In particular, $\cH^{s,\gg}_p(\B)$ is a Banach algebra whenever $s>(n+1)/p$ and $\gg\ge (n+1)/2$.
\end{corollary}

\Proof{Since $s>(n+1)/p$, $H^s_p(\R^{1+n})$ is a Banach algebra. We can therefore assume
$u$ and $v$ to be supported close to the boundary. Then
\begin{eqnarray*}
\|e^{(\gg-\frac{n+1}2)t} u(e^{-t},y)v(e^{-t},y)\|_{H^s_p(\R^{n+1})}
&\le&\|e^{(\gg\frac{n+1}2)t} u(e^{-t},y)\|_{H^s_p(\R^{n+1})}\|v(e^{-t},y)\|_{H^s_p(\R^{n+1})}\\
&\sim& \|u\|_{\cH^{s,\gg}_p(\B)} \|v\|_{\cH^{s,(n+1)/2}_p(\B)},
\end{eqnarray*}
where $\sim$ denotes equivalence of norms. 
}

\begin{corollary}\label{c1} Let $1\le p<\infty$ and $s>(n+1)/p$. 
Then a function $u$ in 
$\cH^{s,\gg}_p(\B)$ is continuous on $\B^\circ$, and, near $\partial\B$, we have
\begin{eqnarray*}
|u(x,y)|\le c x^{\gg-(n+1)/2} \|u\|_{\cH^{s,\gg}_p(\B)}
\end{eqnarray*}
for a constant $c>0$.
\end{corollary}

This is \cite[Corollary 2.5]{RS}. For convenience, we recall the easy proof.
The Sobolev embedding theorem implies continuity as  
$\cH^{s,\gg}_p(\B)\hookrightarrow H^s_{p,loc}(\B^\circ)$.
Near the boundary, we deduce from \eqref{norm} and the trace 
theorem that for each $t\in\R$, 
$$e^{(\gg-(n+1)/2)t}\|u(e^{-t},\cdot)\|_{B^{s-1/p}_{p,p}(\partial\B)}
\le c \|u\|_{\cH^{s,\gamma}_p(\B)}.
$$
For $x=e^{-t}$ we obtain the assertion from the fact that the Besov space 
$B^{s-1/p}_{p,p}(\partial\B)$ embeds into the Sobolev space
$H_p^{s-1/p-\gve}(\partial\B)$ for every $\gve>0$ and the 
Sobolev embedding theorem. \hfill $\Box$

\begin{corollary}\label{c2} Let $1< p<\infty$,  $s\ge0$, $\gve>0$.  
Then the operator $M_m$ of multiplication by a function $m$ in 
$\cH^{s+(n+1)/p+\gve,(n+1)/2}_p(\B)$ defines a continuous map
$$M_m: \cH^{s,\gg}_p(\B)\to \cH^{s,\gg}_p(\B).$$
\end{corollary}

\Proof {Let $v\in \cH^{s,\gg}_p(\B)$ and denote by $\sim$ equivalence of norms. We can assume that $v$ is supported in a single coordinate neighborhood and work in local coordinates. Then
\begin{eqnarray*}
\lefteqn{\|mv\|_{\cH^{s,\gg}_p(\B)} 
\sim \|e^{(\gg-(n+1)/2)t}m(e^{-t},y)v(e^{-t},y)\|_{H^s_p(\R^{n+1})}}\\
&\le& c_1 \|m(e^{-t},y)\|_{C^{s+\gve}_*(\R^{n+1})}
	\|e^{(\gg-(n+1)/2)t}v(e^{-t},y)\|_{H^s_p(\R^{n+1})}\\
&\le& c_2\|m(e^{-t},y)\|_{H^{s+(n+1)/p+\gve}_p(\R^{n+1})}
	\|e^{(\gg-(n+1)/2)t}v(e^{-t},y)\|_{H^s_p(\R^{n+1})}\\
&\sim&\|m\|_{\cH^{s+(n+1)/p+\gve,(n+1)/2}_p(\B)} \|v\|_{\cH^{s,\gg}_p(\B)}.
\end{eqnarray*} 
Here, the first inequality is due to the fact that multiplication by functions in the Zygmund space
$C^\tau_*$ defines a bounded operator in $H^s$ for $-\tau<s<\tau$, cf.\ \cite[Section 13, Theorem 9.1]{Tay}, and the second is a consequence of the fact that $H^{s+(n+1)/p}_p(\R^{n+1})
\hookrightarrow C^s_*$, cf.\ \cite[Section 13, Proposition 8.5]{Tay}.}

%%%%%%%%%%%%%%%%%%%%%%%%%%%%%%%%%%%%%%%%%%%%%%

\section{Bounded Imaginary Powers of Cone Differential Operators}
\setcounter{equation}0

\subsection{Cone differential operators}\label{cdop}
We consider a cone differential operator 
\begin{eqnarray}\label{a}
A: C_c^\infty(\B^\circ,E)\to C_c^\infty(\B^\circ,E)
\end{eqnarray}
of the form  \eqref{conediffop}, acting on sections of a vector bundle $E$ over $\B$.
We may assume that $E$ respects the product structure near the boundary $\partial \B$, i.e. is the pull-back of a vector bundle $E|_{\partial \B}$ over $\partial \B$ 
under the canonical projection $[0,1[\times \partial \B\to \partial \B$. 
In order to keep the notation simple, we shall not indicate the bundles in the function spaces and write $C^\infty_c(\B^\circ), \cH^{s,\gg}_p(\B)$, etc.. 
We moreover assume $A$ to be $\B$-elliptic as explained around \eqref{rps} in the introduction. 

The conormal symbol  of a $A$ is the operator polynomial 
$$\gs_M(A):\C\to \cL(H_p^{s+\mu}(\partial\B), H^s_p(\partial\B)),\text{ defined by } 
\gs_M(A)(z) =   \sum_{j=0}^\mu a_j(x,y,D_y)z^j.$$
To simplify matters, we shall assume:
\begin{eqnarray}\label{conormal}
\gs_M(A) \text{ is invertible on the line }\Re z = (n+1)/2-\gg-\mu
\end{eqnarray}  
(note that the invertibility is independent of $s\in\R$ and $1<p<\infty$).
This implies the existence of a parametrix $B:\cH^{s,\gg}_p(\B)\to 
\cH^{s+\mu,\gg+\mu}_p(\B)$, such that $BA-I$ is compact on $\cH^{s+\mu,\gg+\mu}_p(\B)$ and
$AB-I$ is compact on $\cH^{s,\gg}_p(\B)$.

We shall now consider $A$ as an unbounded operator in 
$\cH^{s,\gg}_p(\B)$. 
Under assumption \eqref{conormal} the domain of the minimal extension $A_{\min,s}$ of $A$, i.e. the closure of $A$ acting as in \eqref{a} in $\cH^{s,\gg}_p(\B)$, is   
$$\cD(A_{\min,s})=\cH^{s+\mu,\gg+\mu}_p(\B),$$ 
and the domain of the maximal extension $A_{\max,s}$ is 
\begin{eqnarray}\label{Dmax}
\cD(A_{\max,s}) =\cD(A_{\min,s}) \oplus \cE,
\end{eqnarray}
where $\cE$ is a finite-dimensional space  of functions of the form 
\begin{eqnarray}
\label{sing}
\sum_{j,k}c_{jk}(y)\go(x) x^{-q_j}\log^kx
\end{eqnarray}
with $c_{jk}\in C^\infty(\partial \B)$, a cut-off function $\go$, 
\begin{eqnarray}\label{real}
\frac{n+1}2-\gg-\mu < \Re q_j < \frac{n+1}2-\gg,\quad \text{and }\ k\in \N_0.
\end{eqnarray}
As a subset of $\cH^{\infty,\gg}_p(\B)$, $\cE$ is independent of $s$.
For details see Section 3 in \cite{GKM} (the statements there are made for the case $p=2$, but they extend to $1<p<\infty$) in connection with \cite[Sections 2.2, 2.3]{Sh}.  

It follows that the domain of an arbitrary closed extension $\underline A_s$ of $A$
in $\cH^{s,\gg}_p(\B)$ is of the form 
$$\cD(\underline A_s) = \cD(A_{\min,s})\oplus \underline \cE
$$
with a subspace $\underline \cE$ of the above space $\cE$. 

\begin{remark}\rm
In case condition \eqref{conormal} does not hold, the minimal domain of $A$, 
considered as an unbounded operator in $\cH^{s,\gg}_p(\B)$ with representation 
\eqref{conediffop} is given by
\begin{eqnarray}\label{Dmin}
\cD(A_{\min,s})= \Big\{ u\in \bigcap_{\gve>0}\cH_p^{s+\mu,\gg+\mu-\gve}(\B): 
x^{-\mu}\sum_{j=0}^\mu a_j(0,y,D_y)(-x\partial_x)^ju\in \cH^{s,\gg}_p(\B)\Big\};
\end{eqnarray}
the maximal domain still  is as in \eqref{Dmax}.
\end{remark}

\subsection{The results on $\cH^{0,\gg}_p(\B)$}

We denote by  $\underline A_0$ a closed extension of $A$, 
considered as an unbounded operator on $\cH^{0,\gg}(\B)$. 
It is the main result of \cite{CSS2} that $\underline A_0$ has bounded imaginary 
powers, provided the resolvent has a certain  structure which we will now recall:

For given $\gt$ and $\gd$, let
$$\gL_\gt = \{z\in \C:|z|\le\gd \text{ or }|\arg z| \ge \gt  \}.$$
Assume that the resolvent to $\underline A_0$ exists on $\gL_\gt$ and 
is of the following form:  
\begin{eqnarray}\label{res}
(\underline A_0-\gl)^{-1} = \go_1(x^\mu\op_M^\gg(g)(\gl)  + G(\gl))\,\go_2 + (1-\go_1)P(\gl) (1-\go_3) 
+G_\infty(\gl).
\end{eqnarray}
Here $\go_1, \go_2, \go_3$ are cut-off functions with $\go_1\go_2 = \go_1$, $\go_3\go_1 = \go_3$, and 
\begin{enumerate}\renewcommand{\labelenumi}{(\roman{enumi})}
\item $g(x, z, \gl) = \tilde g(x, z, x^\mu\gl)$  with 
$\tilde g\in C^\infty(\R_{\ge0}, M_{\cO}^{-\mu,\mu}(\partial \B; \gL_\gt))$
\item $P (\gl) \in  L^{-\mu,\mu}_{cl}(\B^\circ;\gL_\gt )$
\item $G(\gl) \in  R_G^{-\mu,\mu}(\R\times\partial \B; \gL_\gt, \gg )$
\item $G_\infty(\gL_\gt) \in   C_G^{-\infty}(\B;\gL_\gt,\gg)$.
\end{enumerate}
Let us give a short description of these operator classes; full details can be found in 
\cite{CSS2}.
\begin{itemize}
\item $M_{\cO}^{-\mu,\mu}(\partial \B; \gL_\gt)$ is the class of holomorphic Mellin symbols of order 
$-\mu$, depending on the parameter $\gl\in \gL_\gt$ in such a way that if we write $\gl=\gs^\mu$, 
then  $\gs$ plays the role of an additional covariable. 
\item $L^{-\mu,\mu}_{cl}(\B^\circ;\gL_\gt )$ 
denotes the classical pseudodifferential operators of order $-\mu$
depending on the parameter $\gl\in \gL_\gt$ in the same sense as above.
\end{itemize}

We need more notation to explain the other two classes. 
By $C^{\infty,\gg}(\B)$ denote the space of all smooth functions $u$ on $\B^\circ$ such that, in local coordinates near the boundary 
\begin{eqnarray}\label{sgg}
\sup_{0<x<1} x^{(n+1)/2-\gg} |||\log^l x(x\partial_x)^k u(x,\cdot)|||<\infty \text{ for all }	k,l\in \N_0	
\end{eqnarray}
for each semi-norm $ ||| \cdot |||$  of $C^\infty(\partial \B)$. 
Similarly, $\cS_0^{\gg}$  
is the space of all 
$u=u(x,y) \in C^\infty (\R_{>0}\times \partial \B)$ which are rapidly decreasing as $x\to \infty$  and satisfy \eqref{sgg}. 
Now:
\begin{itemize}
\item $R_G^{-\mu,\mu}(\R_{>0}\times \partial \B; \gL_\gt, \gg )$ is the space of 
all operator families  $G = G(\gl)$ with an integral kernel with respect to the 
$L^2(\R_{>0}\times \partial \B, x^ndxdy)$-scalar product of the form
\begin{eqnarray}\label{kernel}
k(\gl,x,y,x',y' ) 
= [\gl]^{(n+1)/\mu}\widetilde k(\gl,[\gl]^{1/\mu} x,y,[\gl]^{1/\mu} x' ,y' ),
\end{eqnarray}
where $[\cdot]$ is a smoothed norm-function (i.e., $[\cdot]$ is smooth, positive on $\C$ and 
$[\gl] = |\gl|$ for large $\gl$), and, for some $\gve = \gve(G)>0$,
\begin{eqnarray}\label{kernel2}
 \widetilde k(\gl,x,y,x',y')\in 
S^{-1}_{cl}(\gL_\gt)\widehat\otimes_\pi \, \cS_0^{\gg+\gve}\, 
\widehat\otimes_\pi\,  \cS_0^{\gg+\gve}. 
\end{eqnarray}
\item  An operator-family $G_\infty = G_\infty(\gl)$, $\gl\in \gL_\gt$, belongs to 
$C_G^{-\infty}(\B;\gL_\gt,\gg)$, if there exists an $\gve  =\gve(G_\infty)>0$ such that 
$G_\infty(\gl)$ has an integral kernel $k(\gl) = k(\gl,\cdot,\cdot)$ with respect to the $\cH^{0,0}(\B)$-scalar product and
$$
k(\gl,y,y') \in \cS(\gL_\gt,C^{\infty,\gg+\gve}(\B)\widehat\otimes_\pi C^{\infty,\gg+\gve}(\B)).
$$
\end{itemize}

The following theorem follows from \cite[Theorem 2]{CSS2} in combination with Remark 6 there.

\begin{theorem}\label{css2}
Let $\underline A_0$ be an extension of the cone differential operator $A$ on $\cH^{0,\gg}_p(\B)$
with domain 
\begin{eqnarray}\label{css2.1}
\cD(\underline A_0) = \cH^{\mu, \gg+\mu}_p(\B)\oplus\underline \cE,
\end{eqnarray}
with a subspace $\underline \cE$ of $\cE$, whose resolvent is of the form \eqref{res} above. 
Then $\underline A_0$ has bounded imaginary powers on 
$\cH^{0,\gg}_p(\B)$ of angle $\gt$.
\end{theorem}

\subsection{Bounded imaginary powers on higher order Mellin-Sobolev spaces}

We next extend Theorem \ref{css2} to higher order Mellin-Sobolev spaces, keeping $\gg, p$ and the 
space $\underline\cE$ in \eqref{css2.1} fixed.

\begin{theorem}\label{hibip} 
For $s\ge0$ denote by $\underline A_s$ the unbounded operator on $\cH^{s,\gg}_p(\B)$ 
with domain 
$$\cD(\underline A_s) = \cH^{s+\mu, \gg+\mu}_p(\B)+ \underline \cE.$$
Assume that the resolvent of the corresponding extension $\underline A_0$ with domain \eqref{css2.1} 
exists on $\gL_\gt$ and is of the form \eqref{res}. 
Then  $\underline A_s$ has bounded imaginary powers on $\cH^{s,\gg}_p(\B)$ of angle $\gt$. 
\end{theorem}

\Proof{{\em Step} 1. 
First let us check that  the restriction of the resolvent of $\underline A_0$ to 
$\cH^{s,\gg}_p(\B)$ is the resolvent of $\underline A_s$. 
The identities 
$$(\underline A_0-\gl )^{-1}(A-\gl)= I\text{ and } (A-\gl)(\underline A_0-\gl )^{-1}=I$$
hold  on $\cD(\underline A_s) $ and $\cH^{s,\gg}_p(\B)$, respectively. Thus it suffices to show that 
$(\underline A_0-\gl )^{-1}(\cH^{s,\gg}_p(\B))\subseteq \cD(\underline A_s)$. 
So suppose $u\in \cD(\underline A_0)$ and $(A-\gl)u\in \cH^{s,\gg}_p(\B)$. 
We first claim that $u\in \cD(A_{\max,s})$.
Indeed, by \eqref{css2.1}, $u\in \cH^{\mu,\gg}_p(\B)$. If $s\le\mu$, we are done. Otherwise, we
have $u\in \cD(A_{\max,\mu})\subseteq \cH^{2\mu,\gg}_p(\B)$. After finitely many steps we reach the assertion. 
Hence $u$ is in the intersection of $\cD(A_{\max,s})$ and $\cD(\underline A_0)$, which is $\cD(\underline A_s)$. 

{\em Step} 2. We show that for arbitrary $s\ge 0$
$$(\underline A_s-\gl)^{-1} = O(|\gl|^{-1}), \quad \gl\in \gL_\gt.$$
We recall that the result was proven for $s=0$ in \cite[Proposition 1]{CSS2} making only 
use of the representation \eqref{res} of the resolvent. 
Using interpolation, it suffices to treat the case where $s$ is a positive integer 
and to show the corresponding estimates for the four types of operators appearing in 
the resolvent.

Ad (i). We write $M(\gl) = x^\mu \op_M^\gg(g)(\gl)$. 
By $\Gamma_\beta$ we denote the set $\{z\in\C: \Re z=\beta\}$.
For $u\in C^\infty_c(\R_{>0}, C^\infty(\partial\B))$ 
we then have
\begin{eqnarray*}
\lefteqn{(x\partial_x)\left(\op_M^\gg(g)(\gl)\right) (u)(x)}\\
&=& (x\partial_x)\int_{\gG_{\frac{n+1}2-\gg}}\int_0^\infty 
(x'/x)^z \widetilde g(x,z,x^\mu\gl) u(x')\, \frac{dx'}{x'} dz\\
&=& \int_{\gG_{\frac{n+1}2-\gg}}\int_0^\infty (x'/x)^z \left( 
(x\partial_x\widetilde g)(x,z,x^\mu\gl) + \mu x^\mu(\partial_\gl \widetilde g)(x,z,x^\mu\gl)
\right)u(x') \, \frac{dx'}{x'} dz\\
&&+\int_{\gG_{\frac{n+1}2-\gg}}\int_0^\infty (x'/x)^z \widetilde g(x,z,x^\mu\gl) (x'\partial_{x'}u)(x') \, \frac{dx'}{x'} dz.
\end{eqnarray*}  
Hence 
\begin{eqnarray}\label{css2.2}
(x\partial_x)(\go_1M(\gl) \go_2)=(\go_1 M(\gl) \go_2 )(x\partial_x) + \widetilde M(\gl)
\end{eqnarray}
with 
\begin{eqnarray*}
\widetilde M(\gl) &=&  \mu\, \go_1M(\gl)\go_2+ (x\partial_x\go_1) M(\gl)\go_2 + \go_1M(\gl) (x\partial_x\go_2) \\&&+ 
\go_1x^\mu \op_M^\gg((x\partial_x\widetilde g)(x,z,x^\mu\gl) + \mu x^\mu(\partial_\gl \widetilde g)(x,z,x^\mu\gl)) \go_2.
\end{eqnarray*}
The operator family $ \gl \widetilde M(\gl)$ is uniformly bounded on 
$\cH^{0,\gg}_p(\B)$, since it satisfies the same assumptions as $\go_1\gl M(\gl)\go_2$ above.

Moreover, the well-known commutator identity  $[D_{y_j},\op p]= \op(D_{y_j}p)$, valid for a 
pseudodifferential symbol $p=p(y,\eta)$  on $\R^n$, implies that we have in local coordinates
$$D_{y_j}M(\gl) =M(\gl)D_{y_j}+ x^\mu\op_M^\gg(g_1(x,z,x^\mu\gl))$$
for a symbol   $g_1\in C^\infty(\R_{\ge0},M^{-\mu,\mu}_{\cO}(\partial \B,\gL_\gt))$. 
Together with equation \eqref{css2.2} this shows that $\go_1\gl M(\gl)\go_2$ is uniformly bounded on $\cH^{1,\gg}_p(\B)$. 
Iteration then implies boundedness on $\cH^{s,\gg}_p(\B)$ for all $s\in \N$. 

Ad (ii). Away from the boundary, the space $\cH^{s,\gg}_p(\B)$ coincides with the usual Sobolev 
space of order $s$. The uniform boundedness of the operator family 
$\gl (1-\go_1)P(\gl)(1-\go_3)$ then is an immediate consequence of the well-known continuity 
result for pseudodifferential operators.

Ad (iii). In order to show the uniform boundedness of $\go_1\gl G(\gl)\go_2$ 
on $\cH^{s,\gg}_p(\B)$, it is sufficient to prove that $\go_1\gl (x\partial_x)^j D^\ga_y G(\gl)\go_2$ 
is uniformly bounded in $\cL( \cH^{0,\gg}_p(\B))$ whenever $j+|\ga|\le s$.
In order to see this, we simply note that, as a consequence of \eqref{kernel}, 
$(x\partial_x)^jD^\ga_yG(\gl)$ has the integral kernel
$$(x\partial_x)^sD^\ga_y k(\gl, x,y,x',y') = 
[\gl]^{(n+1)/\mu}((x\partial_x )^sD^\ga_y\widetilde k)(\gl, [\gl]^{1/\mu}x,y,[\gl]^{1/\mu}x',y').$$
As 
$$(x\partial_x )^jD^\ga_y\widetilde k\in 
S^{-1}_{cl}(\gL_\gt)\widehat\otimes_\pi \, \cS_0^{\gg+\gve}\, 
\widehat\otimes_\pi\, \cS_0^{\gg+\gve}, 
$$
we obtain the required boundedness from \cite[Proposition 1]{CSS2}.

Ad (iv). Elements of $C_G^{-\infty}(\B;\gL_\gt,\gg)$ clearly have operator norms on 
$\cH^{s,\gg}_p(\B)$ 
which are $O(|\gl|^{-N})$ for arbitrary $N$.  

{\em Step} 3. Finally we prove the required estimate for the purely imaginary powers of  $\underline A_s$. 
For $\Re z<0$, $A^z$ is defined by the Dunford integral
$$\frac i{2\pi}\int_{\cC}\gl^z(\underline A_s-\gl)^{-1}\,d\gl$$
where $\cC$ is a contour around $\gL_\gt$.

We have to show that, as $0<\Re z<1$, the norm of the Dunford integral in 
$\cL(\cH^{s,\gg}_p(\B))$  can be estimated by $Me^{\gt|\Im z|}$, uniformly in $\Re z$.   
To this end we estimate the norms of the four terms that arise by replacing 
the resolvent in the integrand by the terms in (i)-(iv). 

Ad (i).  
It is shown in \cite[Proposition 3]{CSS2} that 
$$g_z(x,y, (n+1)/2 -\gg +i\tau,\eta) 
= \go_1(x) x^\mu\int_{\cC}\gl^z\tilde g(x,y, (n+1)/2 -\gg +i\tau,\eta, x^\mu\gl)\, d\gl
$$ 
is a zero order Mellin symbol with respect to the line $\{\Re z = (n+1)/2-\gg\}$ (an element of $MS^0(\R_{>0}\times\R^n\times \gG_{(n+1)/2-\gg}\times\R^n)$ in the notation of \cite[Section10]{CSS2}).
Moreover, the symbol estimates of $e^{-\gt|\Im z|}g_z $ are proven to be uniform in $-1 \le \Re z < 0$. 
This implies the uniform boundedness of the operator norm also on $\cH^{s,\gg}_p(\B)$ for each $s$.  

Ad (ii). This is immediate from the result for standard pseudodifferential operators. 

Ad (iii). Conjugation by $x^\gg$ reduces the task to the case $\gg=0$. 
By definition $G(\gl)$ then is an integral operator with an integral kernel $k(\gl,x,y,x',y' )$
of the form given in \eqref{kernel} and \eqref{kernel2}.
According to the proof of  \cite[Proposition 2]{CSS2} this is sufficient to establish that the norms of the operators  $K_z$ with kernels
$$
k_z(x,y, x',y') =
\int_{\cC}\gl^z k(\gl,x,y, x',y') \, d\gl
$$
on $\cH^{0,0}_p(\B)$ are $O(e^{\gt|\Im z|})$, uniformly for $-1\le\Re z<0$. 

Now it is easy to infer much better mapping properties: For $s\in \N$, the  norms of $K_z: \cH^{0,0}_p(\B) \to \cH^{s,0}_p(\B)$ can be estimated by the operator norms on $\cH^{0,0}_p(\B)$ of the operators with kernels
$$(x\partial_{x})^jD^\ga_{y} k_z(x,y, x',y') =
\int_{\cC}\gl^z (x\partial_{x})^jD^\ga_{y} k(\gl,x,y, x',y') \, d\gl.$$   
Now  \eqref{kernel} and \eqref{kernel2} imply that $(x\partial_{x})^jD^\ga_{y} k(\gl,x,y, x',y')$ has the same structure as $k$, since $\cS^\gve_0$ is invariant under derivatives $(x\partial_x)^j$ and $D^\ga_y$. This implies the uniform boundedness of the operator norms.

Ad (iv). This follows immediately from the fact that $\gl\mapsto G_\infty(\gl)$ is rapidly decreasing with values in bounded operators on $\cH^{s,\gg}_p(\B)$ for each $s\in \R$.  

Hence we obtain the assertion.
}

\subsection{The model cone operator} For completeness, we recall the definition of the 
{\em model cone operator} $\widehat A$ associated with $A$. In the notation \eqref{conediffop}, it is given by 
\begin{eqnarray*}\widehat A = x^{-\mu} \sum_{j=0}^\mu a_j(0,y,D_y)(-x\partial_x)^j.
\end{eqnarray*}
It naturally acts on the spaces $ \cK^{s,\gg}_p(\R_{>0}\times\partial\B) $, $s,\gg\in \R, 1<p<\infty$. 
In order to introduce these, denote first by $H^s_{p,cone}(\R_{>0}\times\partial\B)$ 
the space of all tempered distributions $u$ on $\R_{>0}\times\partial\B$ which belong to the Sobolev space $H^s_p(\R_{>0}\times\partial\B)$ with respect to the cone metric $g_0=dx^2+x^2h(0)$ on $\R_{>0}\times\partial\B$. 
Then let 
$$\cK^{s,\gg}_p(\R_{>0}\times\partial\B) = \{u\in  \cS'(\R_{>0}\times\partial\B): \go u \in \cH^{s,\gg}_p(\B)
\text{ and } (1-\go)u\in H^s_{p,cone}(\R_{>0}\times\partial\B)\}.$$
Here $\go$ is an arbitrary cut-off function on $[0,1[\times \partial \B$. 

Similarly as for $A$ acting in the spaces $\cH^{s,\gg}_p(\B)$, 
one can study the closed extensions of $\widehat A$ in $\cK^{s,\gg}_p(\R_{>0}\times\partial\B)$. 
Under the assumption \eqref{conormal}, the 
domain of minimal extension $\widehat A_{\min,s}$ on $\cK^{s,\gg}_p(\R_{>0}\times\partial\B)$ is given by 
$\cK^{s+\mu,\gg+\mu}_p(\R_{>0}\times\partial\B)$ and the maximal domain is 
$$\cD(\widehat A_{\max,s})= \cK^{s+\mu,\gg+\mu}_p(\R_{>0}\times\partial\B) \oplus \widehat\cE,$$
where $\widehat \cE$ is a set of singular functions as in \eqref{sing}. 
If the operators $a_j(x,y,D_y)$ are independent of $x$ close to $x=0$, then $\widehat \cE$ 
coincides with the space $\cE$ in \eqref{Dmax}. 
In the general case, there is a 1-1-correspondence between $\widehat \cE$ and $\cE$; 
in particular both spaces have the same dimension.
For more details see \cite{GKM}.

%%%%%%%%%%%%%%%%%%%%%%%%%%%%%%%%%%%%%%%%%%%%%%

\section{Higher Regularity for the Cahn-Hilliard Equation on Manifolds with Straight Conical 
Singularities}\label{straight}
\setcounter{equation}{0}

\subsection{The straight cone Laplacian in $\cH^{s,\gg}_p(\B)$}  \label{sc} 
In this section we assume that in a neighborhood of the boundary, say on $[0,1/2\,[\ \times \partial \B$, our Riemannian metric -- which we now denote by $g_0$ in order to distinguish it from the more general 
Riemannian metric $g$ used above -- is of the form 
$$g_0 = dx^2 + x^2h(0)$$  
with the metric $h$ in \eqref{e2}, so that $g_0$ models a manifold with a straight conical
singularity.

We write $\Delta_{0}$ for the Laplace-Beltrami operator on $\mathbb{B}$ with respect to $g_0$.
Our aim is to show that smoother data $u_0$ in the Cahn-Hilliard equation produce smoother short time solutions.
 
By $0=\gl_0>\lambda_1>\gl_2> \ldots$ we denote the different eigenvalues of $\gD_{h(0)}$. We recall that $\dim\B = n+1$ and that the conormal symbol of $\gD_0$ is the operator-valued polynomial
\begin{eqnarray}
z\mapsto \gs_M(\gD_0)(z) = z^2-(n-1)z + \gD_{h(0)}. 
\end{eqnarray}
It is invertible as an operator $H^2(\partial \B) \to L^2(\partial \B)$, provided
$z\not= q^\pm_j$, $j=0,1,2,\ldots$, where 
\begin{eqnarray}\label{qj}
q_j^\pm = \frac{n-1}2 \pm \sqrt{\left(\frac{n-1}{2}\right)^{2}-\lambda_{j}}.
\end{eqnarray} 
The inverse is given by %ES new
\begin{eqnarray}\label{inverse1}
(\gs_M(\gD_0)(z))^{-1}=\sum_{j=0}^\infty \frac1{(z-q^+_j)(z-q^-_j)}\pi_j
\end{eqnarray}
with the orthogonal projection $\pi_j$ onto the eigenspace $E_j$ of $\gl_j$ in $L^2(\partial\B)$. 
As a pseudodifferential operator, it is then also invertible in $\cL(H^{s+2}_p(\partial \B),
H^s_p(\partial \B))$
for any $s\in\R, 1<p<\infty$.
With $q_j^\pm$ we associate the spaces  
$$\cE_{q_j^\pm}=\{x^{-q_j^\pm}\go(x)e(y):e\in E_j\},$$ 
except for the case, where $j=0$ and $\dim \B=2$
when, according to \eqref{inverse1}, we have a double pole in $q_0^\pm=0$  and let 
$$\cE_0=\{\go(x)e_0(y) + \log x\, \go(x)e_1(y): e_0,e_1\in E_0\}.$$
The asymptotics space in \eqref{Dmax} is then  
$$\cE = \bigoplus \cE_\rho$$
with $\rho$ ranging over the $q_j^\pm$ in $]\frac{n+1}2-\gg-2, \frac{n+1}2-\gg[$.

For arbitrary $s\ge 0$ and $1<p<\infty$ we consider $\gD_0$ as an unbounded operator in 
$\cH^{s,\gg}_p(\B)$.  
We fix a weight $\gg$ with 
\begin{gather}\label{fixgamma}
\frac{n-3}{2}<\gamma
<\min\left\{\frac{n-3}2+\overline \gve, \frac{n+1}{2}\right\},
\end{gather}
where 
\begin{eqnarray}\label{epsilonbar}
\overline \gve = -\frac{n-1}2 + \sqrt{\left(\frac{n-1}{2}\right)^{2}-\lambda_{1}}>0.
\end{eqnarray}
This implies in particular that $\gs_M(\gD_0)(z)$ is invertible whenever $\Re z = (n+1)/2-\gg$. 

By $\underline{\Delta}_{0}$ we denote the closed extension of $\Delta_{0}$ on $\cH^{s,\gg}_p(\B)$ with domain
\begin{gather}\label{dom_sgamma}
\mathcal{D}_s(\underline{\Delta}_{0})=\mathcal{H}_{p}^{s+2,\gamma+2}(\B)\oplus \C,
\end{gather}
where $\C$ stands for the constant functions. In this way, the domain of $\underline\gD_0$ (and hence that of $\underline\gD_0^2$) will consist of bounded functions only and will contain functions that do not vanish at the tip of the cone. 
For more information on the choice of $\gg$ see \cite[Section 2.3 and Proposition 3.1]{RS};
it is a special case of that in \cite[Theorem 5.7]{Sh}. 

We will now show that $\underline \gD_0$ has spectrum in $\R_{\le0}$. To this end we first note that $\gD_0$ is symmetric and bounded from above by zero on $C^\infty_c(\B^\circ)$ with respect to the scalar product on $\cH^{0,0}_2(\B)$ given by \eqref{measure} (for $h\equiv h(0)$) and thus has a Friedrichs extension with spectrum in $\R_{\le0}$.

%ES slightly rephrased
 
\begin{theorem}\label{spectrum}
Let $\underline{\Delta}_0$ be the closed extension \eqref{dom_sgamma}
of $\Delta_0$ in $\cH^{s,\gg}_p(\B)$.
Then $\sigma(\underline{\Delta}_{0})\subseteq\mathbb{R}_{\le0}$.    
\end{theorem}
 
%ES First part of proof rewritten using the Friedrichs extension. 
\begin{proof} %We follow the idea used in the proof of \cite[Theorem 5.7]{Sh}.
%We first note that t
%The spectrum is independent of $s$ and $p$.
According to our choice of $\gamma$, the conormal symbol of $\gl-\gD$ is invertible on the line $\Re z =(n+1)/2-\gg$. 
By \cite[Corollaries 3.3 and 3.5]{ScSe}, $\gl-\gD$ has a parametrix in the cone calculus; 
moreover, it is a Fredholm operator in $\cL(\cH^{s+2,\gg+2}_p(\B),\cH^{s,\gg}_p(\B))$ for 
$1<p<\infty$ with index and kernel independent of $s$ and $p$. The same applies to the formal adjoint, which is also $\B$-elliptic, cf.\ \cite[Theorem 2.10]{ScSe}. 
Thus the  invertibility of
$$\gl-\underline \gD: \cH^{s+2,\gg+2}_p(\B)\oplus \C \to\cH^{s,\gg}_p(\B)
$$ is  independent of $s$ and $p$, and we may assume $s=0$, $p=2$.

We also know from \cite{Sh}, specifically Theorem 4.3 in connection with Theorem 5.7,  
that $\gl-\underline\gD_{0}$ is invertible with compact inverse for sufficiently large $\gl$ outside any sector containing $\R_{\le0}$. 
Hence any point in the spectrum is an  eigenvalue.

Let $\underline\gD_{0,F}$ be the Friedrichs extension of $\gD_0$. 
According to \cite[Corollary 5.4]{Sh}, its domain is
\begin{eqnarray}\label{Friedrichs}
\cD(\underline\gD_{0,F}) = \begin{cases}
\cD(\gD_{0,\min}) \oplus \bigoplus_\rho \cE_\rho\oplus \C,&\dim\B=2\\
\cD(\gD_{0,\min}) \oplus \bigoplus_\rho \cE_\rho,&\dim\B>2,
\end{cases}
\end{eqnarray}
where $\gD_{0,\min/\max}$ refers to the minimal/maximal extension on $\cH^{0,0}_2(\B)$ and the summation is over all $\rho$  in $]-1, 0[$\ in the first case and in $](n-3)/2,(n-1)/2]$ in the second. 

We note that $\C$ 
is always contained in $\cD(\underline\gD_{0,F})$. Indeed, for $n=1$ this is trivial. 
For $n=2$, $\C \subseteq \cE_0$, and for $n\ge3$, it is a subset of $\cD(\gD_{0,\min})$. 
Moreover,  $\cD(\underline\gD_{0,F})$ contains $\cD(\gD_{0,\max})\cap \cH^{2,1}_2(\B)$.

Let $\gl\notin \R_{\le0}$, $u\in \cD(\underline\gD_{0})$ and $(\gl-\gD_0)u=0$. 
Write $u = u_0+\mu$ with $u_0\in\cH^{2,\gg+2}_p(\B)$ and $\mu\in\C$. 
Then 
$\gD_0u_0-\gl u_0=\gl \mu\in \cH^{0,(n+1)/2-\gd}_2(\B)$ for all $\gd>0$. 
Hence $u_0$ belongs to the maximal domain of $\gD_0$ in $\cH^{0,(n+1)/2-\gd}_2(\B)$,
thus to $\cD(\gD_{0,\max})\cap \cH^{2,1}_2(\B)\subseteq 
\cD(\underline\gD_{0,F})$ in view of the fact that $2+\gg>(n+1)/2\ge1$.  
We conclude that $u$ belongs to the domain of the Friedrichs
extension and therefore is zero. 
Hence $\gl-\underline\gD_0:\cH^{2,\gg+2}_2(\B)\oplus\C\to \cH^{0,\gg}_2(\B)$ is invertible and so
is $\gl-\underline\gD_0:\cH^{s+2,\gg+2}_p(\B)\oplus\C\to \cH^{s,\gg}_p(\B)$.
\end{proof}

The following theorem extends the statements in \cite{RS} to the case of higher $s$ and arbitrary $c>0$.

\begin{theorem}\label{elldomain3} Let $c>0$, $\gt\in[0,\pi[$, $\phi>0$ and $\gg$, $\underline{\Delta}_{0}$ be chosen as above.
Then $c-\underline{\Delta}_{0}\in \cP(\gt)\cap\mathcal{BIP}(\phi)$ on $\cH^{s,\gg}_p(\B)$ for any $s\ge0$.
\end{theorem} 

\Proof{It follows from \cite[Theorem 5.7]{Sh} in connection with \cite[Remark 2.10]{RS}
that the resolvent of $c-\underline \gD_0$ on $\cH^{0,\gg}_p(\B)$ has the structure required
in Theorem \ref{hibip} for large $c>0$. As the resolvent moreover exists outside $\R_{\le0}$ by Theorem \ref{spectrum} and $\gt\in[0,\pi[$ is arbitrary, we conclude that $c-\underline \gD_0\in \cP(\gt)$ for all $c>0$. Also, $c-\underline \gD_0\in \mathcal{BIP}(\phi)$ 
by a shift of the  integration contour.
}

\subsection{The associated domain of  $\underline\gD_0^2$} \label{domd2}
We first recall the basic facts from  \cite[Section 3.2]{RS}. 
The conormal symbol of $\gD_0^2$ is given by 
\begin{eqnarray*}
\gs_M(\gD_0^2)(z) = \gs_M(\gD_0)(z+2)\gs_M(\gD_0)(z).
\end{eqnarray*} 
By \eqref{inverse1}, its inverse is 
\begin{eqnarray}\label{inverse}
(\gs_M(\gD_0^2)(z))^{-1} &=&   
\sum_{j=0}^\infty\frac{1}{(z-q_j^+)(z-q_j^-)(z+2-q_j^+)(z+2-q_j^-)}\pi_j.
\end{eqnarray}
It is meromorphic with poles of order $\le 2$ in the 
points  $z=q_j^\pm$ and $z=q_j^\pm-2$. 

The domain of the maximal extension 
of $\gD_0^2$ on $\cH^{s,\gg}_p(\B)$ then is of the form  
$$\cD((\gD_0^2)_{\min,s})\oplus \bigoplus_\rho \tilde\cE_\rho.$$
Here $\cD((\gD_0^2)_{\min,s})\subseteq  \cH^{4,\gg+4-\gve}_p(\B)$ for all $\gve>0$, $\rho$ runs over the poles of $\gs_M(\gD_0^2)$   between 
$\dim \B/2-4-\gg$ and $\dim\B/2-\gg$, and the spaces $\tilde\cE_\rho$ are of the form 
\begin{eqnarray*} 
\tilde\cE_\rho = \{x^{-\rho}\log x\go(x)\, e_1(y)  + x^{-\rho}\go(x)\, e_0(y)\}
\end{eqnarray*}
with $e_0$ and $e_1$ in the corresponding eigenspace $E_j$ (the logarithmic terms will only appear for a double pole in $\rho$).

By assumption $\cD(\underline \gD_0^2)$ is a subset of $\cD(\underline\gD_0)= \cH^{2,\gg+2}_p(\B)\oplus\C$. 
In view of the fact that a function $u(x,y) = x^{-\rho}\log x\go(x)e(y)$ with $0\not=e\in C^\infty(\partial \B)$ belongs to $\cH^{\infty,\gg+2}_p(\B)$ if and only if $\Re \rho < \dim\B/2-\gg-2$,
we obtain: 

\begin{lemma}\label{domdelta2} The domain of $\underline\gD_0^2$ as an unbounded operator
on $\cH^{s,\gg}_p(\B)$  is given by 
\begin{eqnarray*}
\cD(\underline\gD_{0,s}^2)= 
\cD(\underline \gD_{0,\min,s}^2) \oplus \bigoplus_\rho \tilde \cE_\rho\oplus \C
\end{eqnarray*}
with the summation now over $\dim \B/2 -\gg-4< \rho<\dim \B/2-\gg-2$. %ES shortened
\end{lemma}

\subsection{Interpolation spaces}\label{interpol}
We want to treat the Cahn-Hilliard equation in $\cH^{s,\gg}_p(\B)$ for $s\ge0$ with the help
of Clément and Li's Theorem \ref{CL}. 
To this end we have to come to a good understanding of the real interpolation space 
$X_q = (X_1,X_0)_{1/q,q}$ for $X_0= \cH^{s,\gg}_p(\B)$ and 
$X_1= \cD(\underline \gD_{0,s}^2)$.
Let us suppose that 
$$ 2<q<\infty.$$
For real $\eta$ with $1/2<\eta<1-1/q$ and $c>0$ we then have 
\begin{eqnarray}\label{e32}
X_{q}=(X_{0},X_{1})_{1-\frac{1}{q},q}
&\hookrightarrow&[X_{0},X_{1}]_{\eta}
=[\cH^{s,\gg}_p(\B), \cD(\underline{\Delta}_{0,s}^2)]_{\eta}=
[\cD((c-\underline{\Delta}_{0,s})^0),\cD((c-\underline{\Delta}_{0,s})^2) ]_\eta\\
&=&[\cD((c-\underline{\Delta}_{0,s})),\cD((c-\underline{\Delta}_{0,s})^2) ]_{2\eta-1} \nonumber
\end{eqnarray}
by the reiteration theorem for interpolation, cf. \cite[(I.2.8.4)]{Am}, using the boundedness of the imaginary powers. See Section \ref{5.4} for an alternative way.

Since $c-\underline{\Delta}\in\mathcal{BIP}(\phi)$, $\phi>0$, 
we apply  (I.2.9.8) in \cite{Am} and obtain
\begin{gather}\label{e62}
[\cD((c-\underline{\Delta}_{0,s})^{0}),\cD((c-\underline{\Delta}_{0,s})^{2})]_{\eta}
\hookrightarrow\cD((c-\underline{\Delta}_{0,s})^{(1-\eta)0+2\eta})
=\cD((c-\underline{\Delta}_{0,s})^{2\eta}).
%\hookrightarrow \cD(\underline\gD).
\end{gather} 
In particular, as $2\eta>1$, we have 
\begin{eqnarray}\label{embedding}
X_q\hookrightarrow \cD(\underline \gD_{0,s}).
\end{eqnarray}
 
\subsection{The linearized equation}
We write the Cahn-Hilliard equation in the form 
$$\partial_t u + A(u) u = F(u), \quad u(0) = u_0$$
with 
\begin{eqnarray}\label{defA}A(v)u
&=&\underline \Delta_0^{2}u+\underline\Delta _0u-3v^{2}\underline\Delta_0 u \,\,\, \mbox{and}\\
F(u)&=&-6u(\nabla u,\nabla u)_{g_0},\label{defF}
\end{eqnarray}
where
\begin{eqnarray}\label{defnabla}
(\nabla u,\nabla  v)_{g_0} = 
\frac1{x^2}\left((x\partial_x u)(x\partial_x v)
+\sum h^{ij}(0)(\partial_{y^i} u)(\partial_{y^j}v)\right).
\end{eqnarray}
Our next goal is to establish the existence of bounded imaginary powers for $A(u_0)$, $u_0\in X_q$, $q>2$. To this end we use the following lemma. A proof can be found in 
\cite[Lemma 3.6]{RS}.

\begin{lemma}\label{l1}
Let $X_0$ be a Banach space and $T\in\mathcal{P}(\theta)$ in the sense of Definition 
$\ref{sectorial}$ with $\theta\geq\pi/2$. 
Then $T^{2}\in\mathcal{P}(\tilde \theta)$ for $\tilde\theta= 2\theta-\pi$ 
and $(T^{2})^{z}=T^{2z}$ for $z\in \C$.
\end{lemma}
 
\begin{proposition}\label{p1}Let $s\ge0$, $p> (n+1)/2$, and $q>2$. 
For every choice of  
$u_0\in X_q$, $\phi>0$, and $\gt\in[0,\pi[$,
the operator $A(u_0)+c_0I$, 
considered as an unbounded operator in 
$\mathcal{H}^{s,\gamma}_{p}(\mathbb{B})$ with domain 
$\cD(\underline{\Delta}_{0,s}^{2})$ belongs to $\cP(\gt)\cap\mathcal{BIP}(\phi)$ 
for all sufficiently large $c_0>0$.
\end{proposition}

\Proof{This proposition was shown as \cite[Proposition 3.7]{RS} for the base space 
$X_0 = \cH^{0,\gg}_p(\B)$ and $u_0\in L^\infty(\B)$. 
In order for the statement to extend to the present situation we need to prove that 
multiplication by a function in $X_q$ defines a bounded operator on $\cH^{s,\gg}_p(\B)$.
For $q>2$, however, we know from \eqref{embedding} that 
$X_q\hookrightarrow \cD(\underline \gD_{0,s}) =\cH^{s+2,\gg+2}_p(\B)\oplus \C$.
Clearly, multiplication by constants furnishes a bounded operator. 
By assumption $s+2> s+ (n+1)/p$; by \eqref{fixgamma} we have $\gg+2\ge (n+1)/2$.  
Hence Corollary \ref{c2} shows that also the functions in $\cH^{s+2,\gg+2}_p(\B)$ 
define continuous multipliers. This completes the argument.}

\subsection{Short time solutions of the Cahn-Hilliard equation}

\begin{theorem}\label{hireg}
Let $p\geq n+1$, $q>2$.  
Given any $u_0\in X_q$, there exists a $T>0$ and a 
unique solution in
$$u\in L^q(0,T; \cD(\underline\gD_0^2))\cap W^{1}_q(0,T; \cH^{s,\gg}_p(\B))
\cap C([0,T],X_q)$$
solving Equation \eqref{CH1} on $]0,T[$ with initial condition \eqref{CH2}. 
\end{theorem}

\Proof{
We write the Cahn-Hilliard equation in the form $\partial_tu+A(u)u=F(u)$ with $A$ 
and $F$ defined in \eqref{defA}  and \eqref{defF}, respectively, and apply Theorem \ref{CL}. 
We have seen in Proposition \ref{p1} that $A(u_0)+cI$ has $\mathcal{BIP}(\phi)$ for arbitrary 
$\phi>0$ and hence maximal regularity, provided $c\ge0$ is large. 
It remains to check conditions (H1) and (H2); note that (H3) is not required.

Let $U $ be a bounded neighborhood of $u_0$ in  $X_q$ and 
$\tilde x$ a function  which equals $x$ near $\partial\B$, is strictly positive on 
$\B^\circ$ and is $\equiv 1$ outside a neighborhood of $\partial \B$.

Concerning (H1): Let $u_1, u_2\in U$. We have seen in the proof of Proposition \ref{p1}
that multiplication by an element in $\cD(\underline\gD_{0,s})$ defines a bounded operator
on $\cH^{s,\gg}_p(\B)$. Moreover, in view of our assumptions on $p$ and $\gg$, we
see from Corollary \ref{c1} that $\cD(\underline\gD_{0,s})$ is an algebra. Therefore
\begin{eqnarray*}
\lefteqn{\|A(u_{1})-A(u_{2})\|_{\cL(X_{1},X_{0})}
=3\|(u_{1}^{2}-u_{2}^{2})\underline{\Delta}_0\|_{\cL(X_{1},X_{0})}
%\leq c\|(u_{1}^{2}-u_{2}^{2})I\|_{\mathcal{L}(\cD(\underline{\Delta}_0),X_{0})}
}
\\
&\le& %c\|(u_1+u_2)(u_1-u_2)I\|_{\mathcal{L}(\cD(\underline\gD),X_{0})}\leq 
c\|(u_1+u_2)(u_1-u_2)I\|_{\mathcal{L}(\cD(\underline{\Delta}_{0,s}),X_0)}
\le c_1 \|u_{1}+u_{2}\|_{\cD(\underline\gD_{0,s})} \|u_{1}-u_{2}\|_{\cD(\underline\gD_{0,s})}\\
&\leq&
 c_{2}(\|u_1\|_{X_q}+\|u_2\|_{X_q})\|u_1-u_2\|_{X_q}
\le c_3\|u_1-u_2\|_{X_q} 
\end{eqnarray*}
for suitable constants $c, c_{1}, c_{2},$ and $c_{3}$, 
where the last inequality is a consequence of the boundedness of $U$. 

Concerning (H2):  
In view of \eqref{e32} and \eqref{e62} the operators $x\partial_x$ and $\partial_{y^j}$, defined near the 
boundary,  map $X_q$ to $\cH^{s+1+\gve, \gg+2}_p(\B)$ 
for some $\gve>0$. Hence we conclude from \eqref{defnabla} that 
$\nabla u\in \cH^{s+1+\gve, \gg+1}_p(\B)$ for $u\in X_q$.
We let $\theta = (n+1)/2-\gg-1$. 
Then  $\tilde x^\gt \nabla u \in \cH^{s+1+\gve, (n+1)/2}_p(\B)$
and $\tilde x^{-2\gt}u \in \cH^{s, \gg}_p(\B)$ for $u\in X_q$ since $\gg>(n-3)/2$. 

In the following estimate we use the facts that, by Corollary \ref{c0},  
$\cH^{s+1+\gve, (n+1)/2}_p(\B)$ is an algebra and, by Corollary \ref{c2},
multiplication by an element in $\cH^{s+1+\gve, (n+1)/2}_p(\B)$ defines a bounded operator
on $\cH^{s,\gg}_p(\B)$. 
\begin{eqnarray}
\lefteqn{\|F(u_{1})-F(u_{2})\|_{X_{0}} 
= \|6u_1(\nabla u_{1},\nabla u_{1})_{g_0}
   -6u_{2}(\nabla u_{2},\nabla u_{2})_{g_0}\|_{X_{0}}}\label{F}\\
 &\le&6\|u_{1}(\nabla u_{1},\nabla u_{1})_{g_0}
 -u_{2}(\nabla u_{1},\nabla u_{1})_{g_0}\|_{X_0}
 \nonumber
 +6\|u_{2}(\nabla u_{1},\nabla u_{1})_{g_0} 
 -u_{2}(\nabla u_{1},\nabla u_{2})_{g_0}\|_{X_0} 
\nonumber \\
&&   +6\|u_{2}(\nabla u_{1},\nabla u_{2})_{g_0}-u_{2}(\nabla u_{2},\nabla u_{2})_{g_0}\|_{X_{0}}
\nonumber\\
&\le&6\|(u_1-u_2)\tilde x^{-2\gt}\|_{X_0}\ \|\tilde x^\gt\nabla u_1\|_{\cH^{s+1+\gve,(n+1)/2}_p}
\ \|\tilde x^\gt\nabla u_1\|_{\cH^{s+1+\gve,(n+1)/2}_p}
\nonumber\\
&&+6\|u_2\tilde x^{-2\gt}\|_{X_0}\ \|\tilde x^\gt\nabla u_1\|_{\cH^{s+1+\gve,(n+1)/2}_p}
\ \|\tilde x^\gt\nabla (u_1-u_2)\|_{\cH^{s+1+\gve,(n+1)/2}_p}
\nonumber\\
&&+6\|u_2\tilde x^{-2\gt}\|_{X_0}\ \|\tilde x^\gt\nabla u_2\|_{\cH^{s+1+\gve,(n+1)/2}_p}
\ \|\tilde x^\gt\nabla (u_1-u_2)\|_{\cH^{s+1+\gve,(n+1)/2}_p}
\nonumber\\
&\le & c_4(\|u_1\|^2_{X_q}+ \|u_1\|_{X_q}\|u_2\|_{X_q} + \|u_2\|^2_{X_q})\|u_1-u_2\|_{X_q}\nonumber
\end{eqnarray}
for $u_1,u_2$ in $X_q$, with a suitable constant $c_4$. As $U$ is bounded, the Lipschitz continuity of $F$ follows.
 }

\section{Higher Regularity on Warped Cones}
\setcounter{equation}{0}

Let now $\Delta$ be as in \eqref{Delta}, the Laplacian on $\B$ induced by the warped cone metric \eqref{e2}. In this case,  the results in \cite{Sh} are no longer applicable. 
We will instead infer maximal regularity of the operator $A(u_0)$ in Equation \eqref{defA} from the results for straight cone Laplacians and perturbation theory for $R$-sectorial operators. 

In addition to $g$ we therefore choose a metric $g_0$ which coincides with $g$ outside 
the collar neighborhood $[0,1]\times \partial \B$ and is of the form 
$$g_0 = dx^2 + x^2 h(0)$$
on $[0,1/2]\times \partial\B$. 
As above we fix $\gg$ according to \eqref{fixgamma} and denote by $\gD_0$ the Laplace-Beltrami operator with respect to  $g_0$ and by
$\underline \gD_0$ the extension in $\cH^{s,\gg}_p(\B)$ with domain 
$$\cD(\underline \gD_{0})= \cH^{s+2,\gg+2}_p(\B)\oplus \C.$$

\subsection{The choice of the extension of $\gD$}\label{choice}

It is clear that the constant functions also belong to the maximal domain of $\gD$ on 
$\cH^{s,\gg}_p(\B)$. We can therefore study the extension $\underline \gD$ of $\gD$ 
with the domain 
\begin{gather}\label{d2} 
\cD(\underline\gD)=\cH^{s+2,\gg+2}_p(\B)\oplus \C.
\end{gather}

Next fix a cut-off function $\go_1$. 
For $0<\varepsilon\le 1$ let  $\omega_{\varepsilon}(t) =\go_1(t/\gve)$.
This is  a cutoff function with support in $[0,\varepsilon)\times \partial\B$.
We define
\begin{gather*}
\Delta_{\varepsilon}=\omega_{\varepsilon}\Delta_{0}+(1-\omega_{\varepsilon})\Delta
\end{gather*}
and consider the closed extension $\underline{\Delta}_{\varepsilon}$ of $\Delta_{\varepsilon}$ with domain given by \eqref{d2}.

Clearly, each of the operators $\underline{\Delta}_{0}$, $\underline{\Delta}$ and $\underline{\Delta}_{\varepsilon}$ induces (by pointwise action) a closed operator in $L^{q}(0,T;\mathcal{H}^{s,\gamma}_{p}(\B))$, $1<q<\infty$, with domain $L^{q}(0,T;\mathcal{H}_{p}^{s+2,\gamma+2}(\B)\oplus\C)$. 

\subsection{R-boundedness and maximal $L^{p}$-regularity}

It follows from \eqref{Delta} that 
\begin{gather}
\Delta=\Delta_{0}+\frac{\partial_{x}|h|}{2x|h|}(x\partial_{x})+\frac{1}{x^{2}}\big(\Delta_{h(x)}-\Delta_{h(0)}\big),
\end{gather}
where $|h|=\mathrm{det}(h_{ij}(x))$. Hence
\begin{gather*}
\Delta-\Delta_{\varepsilon}=B_{\varepsilon},
\end{gather*}
where
\begin{gather*}
B_{\varepsilon}=\omega_{\varepsilon}(x)\ \Big(\frac{\partial_{x}|h|}{2x|h|}(x\partial_{x})+\frac{1}{x^{2}}\big(\Delta_{h(x)}-\Delta_{h(0)}\big)\Big)
\in\mathcal{L}(\mathcal{D}(\underline{\Delta}),\mathcal{H}_{p}^{s,\gamma}(\mathbb{B})).
\end{gather*}

\begin{lemma}\label{bgve}
$B_\gve\to 0$ in $\cL(\cD(\underline \gD), \cH^{s,\gg}_p(\B))$ as $\gve\to 0^+$. 
\end{lemma}

\Proof{The smoothness of $h$ implies that $B_\gve = x\go_\gve(x)C$, where $C$ is a second 
order cone differential operator. It maps the constants to zero; hence it suffices to show 
that $B_\gve$ tends to zero in $\cL(\cH^{s+2,\gg+2}_p(\B), \cH^{s,\gg}_p(\B))$ or, even simpler,  that the norm of multiplication by $x\go_\gve(x)$ tends to zero in $\cL(\cH^{s,\gg}_p(\B))$. To see the latter, we write 
$$x\go_\gve(x)= \gve \ \frac x\gve \,\go\Big(\frac x\gve\Big),$$
note that $(x/\gve)\go(x/\gve)$ is uniformly bounded with respect to $\gve$ in $\cH^{k,(n+1)/2}_p(\B)$ for each $k\ge0$, and infer from Corollary \ref{c2}  that its norm as a multiplier is uniformly bounded. The factor $\gve$ then yields the assertion.  }

The extensions $\underline \gD,\underline \gD_0$ and $\underline\gD_\gve$ 
all have the same model cone operator, namely
$$\widehat \gD = \frac1{x^2}\left((x\partial_x)^2 -(n-1)(-x\partial_x) + \gD_{h(0)}\right).
$$
We consider the extension $\underline{\widehat \gD}$ in $\cK^{0,\gg}_p(\R\times\partial\B)$ 
with domain 
$$\cD(\underline{\widehat\gD}) = \cK^{2,\gg+2}_p(\R\times \partial \B)\oplus \C.$$
It was shown in \cite[Theorem 5.7]{Sh} that 
$$\gl-\underline{\widehat \gD}: 
\cK^{2,\gg+2}_p(\R\times \partial \B)\oplus \C \longrightarrow 
\cK^{0,\gg}_p(\R\times\partial\B)$$
is invertible for $\gl\notin \R_-$. According to  \cite[Theorem 4.2]{Sh} the invertibility is 
independent of $1<p<\infty$; it holds in particular for $p=2$. 
Moreover, it was pointed out shortly before Theorem 4.2 in \cite{Sh} that 
the inverse is given as the sum of two principal edge symbols. 
They are parameter-dependent operators, and hence
$$\|(\gl-\underline{\widehat \gD})^{-1}\|_{\cL(\cK^{0,\gg}_2(\R\times\partial\B))}
= O(|\gl|^{-1}).$$ 
We now apply \cite[Theorem 6.36]{GKM} and conclude that for $\gl\notin \R_-$, $|\gl|$ sufficiently large
\begin{eqnarray}\label{inverse2}
\gl -\underline\gD: \cH^{s+2,\gg+2}_2(\B) \oplus \C \to \cH^{s,\gg}_2(\B)\quad \text{is invertible.}
\end{eqnarray}
(Note that the authors explain in the beginning of Section 3 how to reduce the weight $\gg$ to the fixed choice they use in their formulation of the theorem). 

%The assertion of the following theorem is less evident than it might seem at first glance, 
%since G.\ Mendoza (private communication) has pointed out that, even for domains close to 
%that of the Friedrichs extension, the Laplacian might have spectrum outside $\R_{\le0}$.

\begin{theorem}\label{si}For all $1<p<\infty$, $s\ge0$ and  $\gl\notin \R_-$
\begin{eqnarray}
\gl -\underline\gD: 
\cH^{s+2,\gg+2}_p(\B)\oplus \C\to \cH^{s,\gg}_p(\B)
\end{eqnarray}
is invertible.
\end{theorem}

\Proof{We argue similarly as in the proof of Theorem \ref{spectrum}. 
On the domain where \eqref{inverse2} holds, $\gl-\underline\gD$, has compact resolvent.   
Hence every point in the spectrum necessarily is an eigenvalue.
In order to see that $\gl-\underline\gD$ is actually invertible for all $\gl\notin\R_{\le0}$, 
we first note that we may assume $s=0$ and $p=2$. 
We moreover observe that $\gD$ is symmetric and bounded from above by zero on $C^\infty_c(\B^\circ)$ with respect to the $\cH^{0,0}_2(\B)$-scalar product given by \eqref{measure}. It therefore has a Friedrichs extension $\gD_F$ with spectrum in $\R_{\le0}$. 

The domain of $\gD_F$ has been determined in \cite[Theorem 8.12]{GM}. 
For our  purposes it suffices to know that $\cD(\gD_F)$ contains the constants 
and, by \cite[Lemma 8.1]{GM}, the set $\cD(\gD_{\max})\cap \cH^{2,1}_2(\B)$, where $\gD_{\max}$ is the maximal extension of $\gD$ on $\cH^{0,0}_2(\B)$. 
With this information, the proof proceeds as in Theorem \ref{spectrum}.}
 
\begin{lemma}\label{invert}
Given $c>0$ 
there exist $\gd>0$ and $\gve_0>0$ such that 
\begin{eqnarray}
\label{lower}
\|(c-\underline{\Delta}_{\varepsilon})u\|_{\mathcal{H}_{p}^{s,\gamma}(\mathbb{B})}\geq \delta\|u\|_{\mathcal{D}(\underline{\Delta})},
\quad 0<\varepsilon<\gve_0, u\in\mathcal{D}(\underline{\Delta}).
\end{eqnarray}
\end{lemma}

\Proof{The invertibility of  $c-\underline \gD$ implies that there is a constant 
$\gd>0$ such that
\begin{gather*}
\|(c-\underline{\Delta})u\|_{\cH_{p}^{s,\gamma}(\mathbb{B})}\geq 2\delta\|u\|_{\cD(\underline{\Delta})}.
\end{gather*}
Now the assertion follows from Lemma \ref{bgve}.
}

\begin{proposition}\label{bipbound}Let $\phi>0$ and $c>0$. 
By possibly increasing $c$ we can achieve that the $\mathcal{BIP}(\phi)$-bounds of $c-\underline\gD_\gve$ can be estimated uniformly with respect to $\gve$. 
\end{proposition}

\Proof{We will show that, for sufficiently large $\gl\in\gL_\gt$, 
we can choose all components in the representation \eqref{res} of the resolvent of 
$c-\underline\gD_\gve$  to be uniformly bounded
with respect to $\gve$ in the respective seminorms.

Let $\gve_1$ be fixed and $\gve<\gve_1$. 
As already observed in the proof of Lemma \ref{bgve}, the difference  $\gD_{\gve_1}-\gD_{\gve}= B_\gve-B_{\gve_1}$ is of the form  $x(\go_\gve(x)-\go_{\gve_1}(x)) C$ with a second order cone differential operator $C$, independent of $\gve$. 
The support of the difference therefore lies in $\gve<x<\gve_1$. 

Next we note that  the conormal symbol of $\gD_\gve$ is independent of $\gve$.
Moreover, the symbol seminorms of the $x$-coefficient $(\go_\gve(x)-\go_{\gve_1}(x)) C$ are uniformly bounded. 
(For this we measure  $x\partial_x$-derivatives instead of the usual seminorms 
in $C^\infty(\R_{\ge0}, M_{\cO}^{2,2}(\partial\B,\gL_\gt))$, which is sufficient for our purposes, cf.\ \cite[Proposition 3]{CSS2}.)
Hence the symbol inversion process produces a Mellin symbol $g_\varepsilon$ for the parametrix to 
$c-\gD_\gve-\gl$ with uniformly bounded seminorms. 

This has an important consequence: As $\Delta_{\varepsilon_1}$ and 
$\Delta_{\varepsilon_2}$ coincide outside $-\varepsilon<x<\varepsilon_1$, the model cone operators of both coincide, and the pseudodifferential
parts agree for $x>\varepsilon_1$. Write, as in  \eqref{res},
$$(c-\underline\gD_{\varepsilon_1}-\lambda)^{-1} 
= \omega_1\left( \op_M^\gamma (g_{\varepsilon_1})(\gl)+G(\gl)\right)\go_2 
+ (1-\go_1)P(1-\go_3)+G_\infty(\gl)$$
with cut-off functions   
satisfying $\go_1\go_2 = \go_1$ and $\go_1\go_3 = \go_3$.  
In view of the parametrix construction process, see Section 3.2, in particular Eq.\ (3.10),
in \cite{Sh}, we then obtain a parametrix to $c-\underline\gD_{\varepsilon}-\lambda$ by simply replacing $g_{\varepsilon_1}$ with the corresponding Mellin symbol 
$g_{\varepsilon}$ as indicated above. Hence 
$$(c-\underline\gD_{\varepsilon}-\lambda)
\left( \omega_1\left( \op_M^\gamma (g_{\varepsilon})(\gl)+G(\gl)\right)\go_2 
+ (1-\go_1)P(1-\go_3)\right)= I+ G_{\infty,\varepsilon}(\lambda) $$
for a uniformly bounded family $G_{\gve,\infty}\in C^{-\infty}_G(\B,\gL_\gt,\gg)$. 

As $G_{\gve,\infty}(\gl)\to 0$,
$(I+G_{\gve,\infty}(\gl))^{-1}$ exists  as an operator in $\cL(\cH^{0,\gg}_2(\B))$ 
for sufficiently large $\gl$ in $\gL_\gt$, say $|\gl|\ge R$, 
where $R$ is independent of $\gve$. 
Writing 
$$(I+G_{\gve,\infty}(\gl))^{-1} = I+H_\gve(\gl)\ \ \text{with }\  H_\gve(\gl) = G_{\gve,\infty}(\gl)-G_{\gve,\infty}(\gl)(I+G_{\gve,\infty}(\gl))^{-1}G_{\gve,\infty}(\gl)
$$
and noting that 
$$\partial_\gl(I+G_{\gve,\infty}(\gl))^{-1}=-(I+G_{\gve,\infty}(\gl))^{-1}\partial_\gl G_{\gve,\infty}(\gl)(I+G_{\gve,\infty}(\gl))^{-1},$$
we see that, given $k,m\in \N_0$, the seminorms for
$$\gl^k\partial_\gl^m H_\gve(\gl): \cH^{0,\gg}_2(\B) \to C^{\infty,\gg+\gd}(\B)$$
(with suitably small $\gd>0$) are uniformly bounded in $\gve$. 
A similar consideration applies to the adjoints.
We infer from \cite[Corollary 4.3]{sek} that $H_{\gve}$ has an integral kernel in $\cS(\gL_\gt\cap \{|\gl|\ge R\}, C^{\infty,\gg+\gd}(\B)\widehat\otimes_\pi C^{\infty,\gg+\gd}(\B))$ which
satisfies the corresponding estimates uniformly in $\gve$.
As $(I+G_{\gve,\infty}(\gl))^{-1}= I+H_\gve(\gl)$ and 
$$(c-\underline\gD_\gve-\gl)^{-1} 
= (\go_1 \left(x^2\op_M^\gg(g_\varepsilon)(\gl)+G(\lambda)\right)\go_2 + (1-\go_1)P(1-\go_3))(I+G_{\gve,\infty}(\gl))^{-1},
$$
the symbol seminorms for the components of the resolvent to 
$c-\underline\gD_\gve$ with respect to the decomposition in \eqref{res} are independent of the choice of $\gve$ for  $\gl\in \gL_\gt$, $|\gl|\ge R$.
Replacing $c$ by $c+R$, the estimates will hold in all of $\gL_\gt$.
The argument in \cite{CSS2} then shows that the $\mathcal{BIP}(\phi)$-bound is 
uniformly bounded in $\gve$.    
  } 
 
\begin{corollary}\label{lowerbound}
Fix $c$ and $\gve_0$  as in Lemma $\ref{invert}$, with $c$ possibly increased according to
Proposition $\ref{bipbound}$. 
By Theorem $\ref{elldomain3}$, the operators $c-\underline{\Delta}_{\varepsilon}$ belong to $\mathcal{P}(\theta)\cap\mathcal{BIP}(\phi)$, for any 
$\theta\in[0,\pi[$ and $\phi>0$.  
Since $\mathcal{H}_{p}^{s,\gamma}(\mathbb{B})$ is UMD, we deduce from 
\cite[Theorem 4]{CP} that $c-\underline{\Delta}_{\varepsilon}$ 
is $R$-sectorial with  angle $\phi$. Moreover,  Proposition $\ref{bipbound}$ 
shows that the $\mathcal{BIP}(\phi)$-bounds and hence - see the proof of \cite[Theorem 4]{CP} - the $R$-bounds for $c-\underline
\gD_\gve$ are uniformly bounded in $0<\gve<\gve_0$ for sufficiently small $\gve_0$. 
\end{corollary}

\begin{theorem}\label{RB}
For any $\theta\in[0,\pi[$ and $c>0$, 
$c-\underline{\Delta}$ is $R$-sectorial of angle $\theta$. Hence $c-\underline{\Delta}$ generates an analytic semigroup and has maximal $L^{p}$-regularity.
\end{theorem}

\Proof{We first assume that $c>0$ is sufficiently large and infer from \eqref{lower} and Lemma $\ref{bgve}$ that 
\begin{gather}\label{ee4}
\|B_{\varepsilon}u\|_{\mathcal{H}_{p}^{s,\gamma}(\mathbb{B})}\leq c_{\varepsilon}\|u\|_{\mathcal{D}(\underline{\Delta})}\leq \frac{c_{\varepsilon}}{\delta}\|(c-\underline{\Delta}_{\varepsilon})u\|_{\mathcal{H}_{p}^{s,\gamma}(\mathbb{B})},
\end{gather} 
where the constant $c_{\varepsilon}>0$ can be made arbitrarily small by taking $\varepsilon$ sufficiently close to zero. 
In particular, it can be made smaller than the inverse of the supremum of the $R$-bounds of the operators $c-\underline\gD_\gve$, cf.\ Corollary \ref{lowerbound}. From \cite[Theorem 1]{KL}, we infer that $c-\underline\gD$ is $R$-sectorial of angle $\gt$ for large $c>0$.

On the other hand, since the resolvent exists outside $\R_{\le0}$ and $\gt\in[0,\pi[$ is arbitrary, we obtain the $R$-sectoriality for all $c>0$.  
}

As in Section \ref{straight} we now study the linearized equation with $g_0$ replaced by $g$
and $\underline\gD_0$ by $\underline\gD$. 
We use Theorem \ref{RB} to infer maximal regularity for the operator
$A(v)$ given by 
$$A(v)u = \underline\gD^2u + \underline\gD u -3v^2 \underline\gD u ,$$
cf.\  \eqref{defA}:

\begin{proposition}\label{RBPA} 
Let $v\in \cH^{s+(n+1)/p+\gve,(n+1)/2}_p(\B)\oplus\C$ for some 
$\gve>0$. 
For every $\theta\in[0,\pi[$ the operator $A(v) +c_0I$
is $R$-sectorial of angle $\theta$,  provided $c_0>0$ is sufficiently large. 
Hence $A(v)$ has maximal $L^{q}$-regularity, $1<q<\infty$.
\end{proposition}

\begin{proof}
We write
\begin{gather}\label{e12a}
\lambda\big((c-\underline{\Delta})^{2}+\lambda\big)^{-1}
=(i\sqrt{\lambda})(c-\underline{\Delta}+i\sqrt{\lambda})^{-1}
(-i\sqrt{\lambda})(c-\underline{\Delta}-i\sqrt{\lambda})^{-1}.
\end{gather}
As $\arg (\pm i\sqrt\gl) = \frac12\arg\gl \pm\frac12\pi$, this implies that  $(c-\underline{\Delta})^{2}\in\mathcal{P}(\theta)$ for any $\theta\in[0,\pi[$, 
whenever $c>0$ is large enough. Similarly,
$\pm i\sqrt{\lambda}$ will lie in the domain of $R$-boundedness of 
$z(c-\underline{\Delta}+z)^{-1}$. 
Thus, both families 
\begin{gather*}
(i\sqrt{\lambda})(c-\underline{\Delta}+i\sqrt{\lambda})^{-1} \text{ and }\ 
(-i\sqrt{\lambda})(c-\underline{\Delta}-i\sqrt{\lambda})^{-1}
\end{gather*}
are $R$-bounded.  
By the definition of $R$-boundedness, we deduce from \eqref{e12a} that 
$(c-\underline{\Delta})^{2}$ is $R$-sectorial with angle $\gt$.

Next let $\mu>0$. 
The sectoriality of $c-\underline{\Delta}$ implies that the norm of the bounded operator 
$\mu(c-\underline{\Delta})^{-2}$ in $\cL(\cH^{s,\gg}_p(\B))$ 
will become arbitrarily small, provided $c>0$ is chosen sufficiently large. 
Given any $\ga>0$ we will have 
$$\|\mu x\|_{\cH^{s,\gg}_p(\B)} <
\alpha\|(c-\underline{\Delta})^{2}x\|_{\cH^{s,\gg}_p(\B)}, 
\quad x\in \cD((c-\underline\gD)^2),$$
for sufficiently large $c>0$.
Thus, by Theorem 1 in \cite{KL}, $(c-\underline{\Delta})^{2}+\mu$ is $R$-sectorial with angle $\gt$ for large  $c>0$.

For any $f\in \cH^{s+(n+1)/p+\gve,(n+1)/2}_p(\B)\oplus\C$ 
we have (with norms taken in $\cL(\cH^{s,\gg}_p(\B))$)
\begin{eqnarray*}
\lefteqn{\|f\underline{\Delta}\big((c-\underline{\Delta})^{2} +\mu\big)^{-1}\|
=\|f\underline{\Delta}(c-\underline{\Delta}+i\sqrt{\mu})^{-1}
(c-\underline{\Delta}-i\sqrt{\mu})^{-1}\|}\\
&\leq&\|M_f\|\|(c-\underline{\Delta}+i\sqrt{\mu}-i\sqrt{\mu}-c)
(c-\underline{\Delta}+i\sqrt{\mu})^{-1}\|
\|(c-\underline{\Delta}-i\sqrt{\mu})^{-1}\|\\
&=&\|M_f\|\|I-(i\sqrt{\mu}+c)(c-\underline{\Delta}+i\sqrt{\mu})^{-1}\|
\|(c-\underline{\Delta}-i\sqrt{\mu})^{-1}\|.
\end{eqnarray*}
Here, we have written $M_f$ for the operator of multiplication by $f$. 
By Corollary \ref{c2}, it is bounded on $\cH^{s,\gg}_p(\B)$. 

We can make the last term in the above inequality arbitrarily small by taking $\mu$ large.
Hence, given an arbitrary $\ga>0$, we can take $c>0$ so large that
$$\|f\underline{\Delta}x\|<\alpha\|\big((c-\underline{\Delta})^{2}+\mu\big)x\|, \quad
x\in\mathcal{D}((c-\underline{\Delta})^{2}).$$  
Another application of  \cite[Theorem 1]{KL} then furnishes that 
$$\underline \gD^2 +(f-2c)\underline \gD + (c^2+\mu)$$ is $R$-sectorial with angle $\gt$ for
$c$ and $\mu$ large. 
Given $v\in  \cH^{s+(n+1)/p+\gve,(n+1)/2}_p(\B)\oplus\C$, the assertion follows for $f=2c+1-3v^2$, 
noting that the latter space is an algebra by Corollary \ref{c0}.
\end{proof}

\subsection{The domain of $\underline\gD^2$} 

Let $\cD(\underline\gD)$ be as in \eqref{d2} and
$$\cD(\underline\gD^2) = 
\{ u\in \cD(\underline\gD): \gD u \in \cD(\underline\gD)\}.$$
In view of \eqref{Dmax}, it is of the form
\begin{eqnarray*} 
\cD(\underline \gD^2) = 
\cD(\underline\gD_{\min,s}^2)\oplus \bigoplus \cF_\rho \oplus \C
\end{eqnarray*}
with $\cD(\underline\gD_{\min,s}^2) \subseteq \cH^{s+4,\gg+4-\gve}_p(\B)$ for every $\gve>0$,
and asymptotics spaces $\cF_\rho$, which can be determined explicitly from the
conormal symbol of $\gD^2$ as well as  the metric $h$ in \eqref{Delta} and its derivatives at $x=0$. %ES2
For details see e.g.\ \cite[Section 6]{GKM},  \cite[Section 2.3]{Sh} or \cite{se}. 
Here it suffices to say that $\rho$ varies over the non-invertibility points of the conormal
symbol of $\gD^2$ in $](n+1)/2-\gg-4, (n+1)/2-\gg[$, 
%(apart from $\gs=0$ which is taken care of by the summand $\C$), 
and the elements in each $\cF_\rho$ are suitable linear combinations of 
functions of the form $x^{-\rho+j}\log^kx\go(x)e(y)$, where $j\in\{0,1,2,3\}$, $k\in\{0,1\}$ 
and $e\in C^\infty(\partial \B)$.  
The space $\C$ has been listed separately, because, as $\cD(\underline\gD^2)\subseteq 
\cD(\underline\gD)$, we may conclude that 
%In view of the fact that $\cD(\underline\gD^2)\subseteq\cD(\underline\gD)$, 
the $\cF_\rho$ are subsets
of $\cH^{s+2,\gg+2}_p(\B)$. 
In particular,  $\cD(\underline\gD^2) \subseteq \cH^{s+4,(n+1)/2-\gve}_p(\B)$ for every $\gve>0$.

\subsection{Short time solutions of the Cahn-Hilliard equation}\label{5.4}

In order to determine the interpolation space $X_q = (X_1,X_0)_{1/q,q}$ between 
$X_0 = \cH^{s,\gg}_p(\B)$ and $X_1 = \cD(\underline\gD^2)$ we invoke the following result
of Haase \cite[Corollary 7.3]{Haase}:
\begin{theorem}\label{HInt}
Let $A$ be a sectorial operator  on the Banach space $X$, 
and let $\ga,\gb,\gg\in \C$ with $ 0<\Re\gg<\Re\gb\le\Re\ga$, $\gs\in(0,1)$, 
$q\in[1,\infty]$,and $x\in X$. Then 
$$x \in (X,\cD(A^\ga))_{\gt,q} \Rightarrow x \in \cD(A^\gg) \text{ and } 
A^\gg x \in (X,\cD(A^{\gb-\gg}))_{\gs,q},$$
where $\gt=(1-\gs)\frac{\Re\gg}{\Re \ga} +\gs\frac{\Re\gb}{\Re \ga}$.
Moreover, 
\begin{eqnarray}\label{HReit}
(X,\cD(A^\ga))_{\gt,q} = (\cD(A^\gg),\cD(A^\gb))_{\gs,q}.
\end{eqnarray}
\end{theorem}

For $q>2$ we apply the above theorem for $A=\underline \gD$ on $X_0$ 
with $\ga=2$, $\gs=1/2$ and $1<\gg<\gb\le2$
so small that the resulting $\gt$ satisfies  $\frac12<\gt\le1-\frac1q$.
We conclude that 
\begin{eqnarray}\label{Hin}
X_q\hookrightarrow (X_0, \cD(\underline\gD^2))_{\gt,q}\hookrightarrow 
(\cD(\underline\gD^\gg), \cD(\underline\gD^\gb))_{\gs,q} \hookrightarrow
\cD(\underline\gD^\gg)\hookrightarrow \cD(\underline\gD)= \cH^{s+2,\gg+2}_p(\B)
\oplus \C.
\end{eqnarray}

%We obtain with 
%\cite[Lemma 5.4]{CSS1}: 
%\begin{eqnarray*}
% X_q\hookrightarrow (\cH^{s+4,(n+1)/2-\gve}_p(\B),\cH^{s,\gg}_p(\B))_{1/q,q}
% \hookrightarrow \cH^{\tilde s, \tilde \gg}_p(\B),
%\end{eqnarray*}
%for every $\gve>0$, every $\tilde s < s+4(1-1/q)$ and every $\tilde \gg <  (1-1/ q)((n+1)/2-\gve)+\gg/q$.
%For arbitrary $\tilde q<q$ we thus have
%\begin{eqnarray}\label{interpol2}
%X_q \hookrightarrow \cH^{s+4(1-1/\tilde q),(n+1)/2-((n+1)/2-\gg)/\tilde q }_p(\B).
%\end{eqnarray}

We then obtain a complete analog of Theorem \ref{hireg}:
\begin{theorem}\label{hireg2}
Let $s\ge0$ and  $\underline \gD$ as explained in Section $\ref{choice}$. 
Choose $p\geq n+1$ and $q>2$.
Given any $u_0\in X_q$, there exists a $T>0$ and a 
unique solution in
$$u\in L^q(0,T; \cD(\underline\gD^2))\cap W^{1}_q(0,T; \cH^{s,\gg}_p(\B))
\cap C([0,T],X_q),$$
solving the Cahn-Hilliard equation \eqref{CH1} on $]0,T[$ with initial condition $u(0)=u_0$. 
\end{theorem}

\begin{proof} By Proposition \ref{RBPA}, $A(u_0)$ has maximal regularity. 
It remains to establish properties (H1) and (H2). 
 
As $p\ge n+1$ and $\gg+2>\frac{n+1}2$, $\cD(\underline\gD)$ is an algebra by Corollary \ref{c1}.
Moreover, \eqref{Hin} implies that $X_q$ embeds into $\cH^{s+(n+1)/p+\gve,(n+1)/2}_p(\B)\oplus \C$ for some $\gve>0$.
By Corollary \ref{c2}, multiplication by functions in $X_q$ defines continuous operators on $\cH^{s,\gg}_p(\B)$.
With this information, (H1) and (H2) can be verified as in the proof of Theorem \ref{hireg}.  
\end{proof}  

\subsection*{Acknowledgment} The authors thank G.\ Mendoza and J.\ Seiler for valuable 
discussions and the referee for very helpful comments which led to an improvement in 
Theorem \ref{hireg2}.

%%%%%%%%%%%%%%%%%%%%%%%%%%%%%%%%%%%

\end{document}